\documentclass[12pt]{article}

\usepackage{times} %The package is now obsolete, replaced by the mathptmx package, which supports Times Roman text and (mostly) matching mathematics.
\usepackage[a4paper,text={128mm,185mm}, centering]{geometry}

\usepackage{changepage}
\usepackage{titlesec}

\titleformat{\section}[hang]%
{\bfseries\large}{\thesection.}{1ex}{}%

\titleformat{\subsection}[hang]%
{\bfseries}{\thesubsection}{1ex}{}%

\usepackage{amsthm}%Pour les remarques
\usepackage{fancyhdr}
\theoremstyle{plain} % texte en italique
\newtheorem{theorem}{Theorem}[section]
\newtheorem{thm} [theorem] {Theorem}

\newtheorem{prop}[theorem] {Proposition}

\newtheorem{df}[theorem]{Definition}

\newtheorem{rmq}[theorem] {Remark}

\theoremstyle{definition} % pas en italique
\newtheorem{example}[theorem]{Example}
\newtheorem{exm}[theorem]{Example} 

\usepackage[utf8]{inputenc}% UTF8 ça ne change rien au niveau du rendu, mais ça permet d'utiliser le codage UTF8

\usepackage{amssymb,amsmath, amsfonts}
\usepackage{mathabx}

\usepackage{MnSymbol}

\usepackage{cjhebrew} %pour les caractères hébraïques

\newcommand*{\twosquig}{%
\mathrel{\vcenter{\offinterlineskip
\hbox{$\rightsquigarrow$}\vskip-.08ex\hbox{$\rightsquigarrow$}}}}

\usepackage{tocloft}

\addtocontents{toc}{\cftpagenumbersoff{chapter}}
\addtocontents{toc}{\cftpagenumbersoff{section}}
\addtocontents{toc}{\cftpagenumbersoff{subsection}}

\title{\vskip 5pt  \bf INTERACTING OPEN DYNAMICS}
\author{\itshape\bfseries {St\'{e}phane DUGOWSON}}
\date{}

\begin{document}
\maketitle

\begin{small}
 {Article to be published in the fall 2020 issue of the \emph{Cahiers de Topologie et Géométrie Différentielle Catégoriques}, vol LXI-4. NB: The pagination of the two versions differs.}
 \end{small}

%Nota bene:  the next two commands erase pages numbers (the right page numbers will be add by the editors on the pdf) 
%\cfoot{}
%\thispagestyle{empty}
%%
\vskip 25pt

% Abstract in French and in English, followed by Keywords and MSC:
\begin{adjustwidth}{0.5cm}{0.5cm}
{\small
{\bf R\'esum\'e.} Cet article présente les concepts de base d'une théorie systémique de l'interaction entre des dynamiques ouvertes non déterministes à temporalités variées. Elle comporte trois niveaux : la définition de ces dynamiques en tant que lax-foncteurs, la notion d'interaction --- qui fait appel à des notions de {requêtes}, de {synchronisations} et de {modes sociaux} --- et enfin l'engendrement de dynamiques globales ouvertes. L'aspect connectif des interactions est abordé, mais les autres aspects connectifs sont renvoyés à des travaux ultérieurs. \\
{\bf Abstract.} This paper  
presents the basic concepts of a systemic theory of interaction between non-deterministic open dynamics with varying temporalities, which includes three stages: the definition of these dynamics as lax-functors, the notion of interaction --- which uses some notions of requests, synchronizations and social modes (privacy)  --- and finally the generation of open global dynamics. Some connectivity structures of an interaction are defined, but the other aspects of dynamical connectivity  are left to further work.
}
\end{adjustwidth}

% Here starts the text:.
\section*{Introduction}\addcontentsline{toc}{section}{Introduction}

This article  
presents in English the fundamental concepts of our theory of interactivity between some open dynamics defined as kind of lax-functors to some $2$-categories of sets with families of non-deterministic transitions as $1$-cells. The origin of this work is linked to our research on connectivity structures \cite{Dugowson:201012}, since connectivity has proven to be essentially dynamic in nature. In 2009, we  began to study from a connectivity point of view some dynamics that were not necessarily deterministic, with durations taken in an arbitrary monoid. 
During a lecture in 2010 on these issues, an oral remark by Mme Andrée Ehresmann suggested that any small categories should be taken as duration systems. On this occasion, she mentioned her 1965 paper \cite{Bastiani_Ehresmann:196505}, where under the name of \emph{guidable systems} she considers kinds of deterministic influenceable dynamical systems based on temporalities defined by small topological categories. 
At the end of the first section of the present paper, we precise some relations between this notion of ``guidable systems" due to Mme Ehresmann and the one we developed on our side after the question of the interaction between our own non-deterministic dynamics based on various temporalities arose. 
In the course of our research on interactivity, we thus first considered open dynamics as defined by some functors said to be \emph{disjunctive} and we sought to construct the global dynamics generated by the interaction of families of such open dynamics. The problem was to recognize that functors were insufficient because of a kind of instability: global dynamics were not always functors. We then had to extend our definitions to what we first called sub-functors (\cite{Dugowson:20150807, Dugowson:20150809,Dugowson:20160831}), before Mathieu Anel and then Mme Ehresmann invited us to reformulate our definitions in terms of lax-functors \cite{Dugowson:20180527}.
 Thanks to the lax-functorial stability theorem, presented at the beginning of the section \ref{sec Dyn Produced}, ``Global Dynamics", we obtain a systemic theory  where the dynamics generated by interactive families can in turn interact.

After this introduction and details on our notations and  the $2$-categories  used in the paper, there are three sections: 
\begin{itemize}
\item in the first section, we define what we call \emph{open dynamics} thanks  to notions of \emph{multi-dynamics}, \emph{mono-dynamics}, \emph{clocks},  and morphisms between them. We also define some \emph{parametric quotients} that are used in the third section. We then give a number of examples and, finally, we briefly describe some of the relations between our dynamics and Mme Ehresmann's guidable systems,
\item in the second section we define precisely what we mean by an \emph{interaction}. This is the only part of the article where we discuss connectivity structures, leaving the other connectivity aspects of dynamical interactivity to further work,
\item finally, in the third section, the lax-functional stability theorem makes it possible to associate a number of global dynamics with a given interactive family, and we conclude the paper with two examples.
\end{itemize}

\subsection*{Notations and $2$-categories at stake}\label{subs* notations}\addcontentsline{toc}{subsection}{Notations and $2$-categories at stake}

\subsubsection*{Functions}%\addcontentsline{toc}{subsubsection}{Functions}

The canonical inclusion $\emptyset \hookrightarrow \mathbf{R}$, that is the only real function defined on the empty set $\emptyset$, is denoted by $\underline{\emptyset}$. The restriction of a function $f$ on a subset $D\subseteq \mathbf{R}$ is denoted $f_{\vert D}$.
For any integer $k\in\mathbf{N}$, we set $\mathcal{C}^k=\bigcup_{D\in \mathcal{I}_\mathbf{R}}{\mathcal{C}^k(D)}$, where $\mathcal{I}_\mathbf{R}$ is the set of open real intervals and, for each interval $D$, $\mathcal{C}^k(D)$ is the set of real functions of class $\mathcal{C}^k$ defined on it\footnote{If $D$ is a singleton, we consider $\mathcal{C}^k(D)$ as the set of constant functions, \emph{i.e.} $\mathcal{C}^k(D)\simeq\mathbf{R}$.}. The set $\mathcal{C}^0$ of all continuous real functions defined on open real intervals is also denoted by $\mathcal{C}$.
Note that $\mathcal{C}^k(\emptyset)=\{\underline{\emptyset}\}\neq\emptyset$. 
For each interval $D\subseteq \mathbf{R}$, the set of metric maps  $D\rightarrow \mathbf{R}$, that is the set of real Lipschitz maps with Lipschitz constant $1$, is denoted $Lip_1(D)$.

\subsubsection*{Categories}

As usual in our papers, for any category $\mathbf{D}$, we denote by $\dot{\mathbf{D}}$ the class of its objects, and $\overrightarrow{\mathbf{D}}$ the class of its arrows. For every arrow $h$, we denote by $\mathrm{dom}(h)$ its source object (or \emph{domain}) and by $\mathrm{cod}(h)$ its target object (or \emph{codomain}).
$\mathbf{1}= (\bullet)$ is the terminal category, which has only one arrow $Id_\bullet$ that is also denoted by $\overrightarrow{0}$. The category of sets is denoted by $\mathbf{Sets}$. The discrete $2$-category associated with any category $\mathbf{D}$ is again denoted as $\mathbf{D}$.

\subsubsection*{Transitions}\label{subsubs* transitions}

For any sets $U$ and $V$, we define a \emph{transition from $U$ to $V$} as a map $U\rightarrow\mathcal{P}(V)$ or, equivalently, as a binary relation $U \rightarrow V$. We often write $\varphi:U\rightsquigarrow V$ to indicate that $\varphi$ is such a transition with $U=\mathrm{dom}(\varphi)$ and $V=\mathrm{cod}(\varphi)$. The \emph{domain of definition} of $\varphi$ is defined by  $\mathrm{Def}_\varphi:=\{u\in U, \varphi(u)\neq\emptyset\}$.
Denoted by $\psi\odot \varphi$, the composition of transitions $\varphi:U\rightsquigarrow V$ and $\psi:V\rightsquigarrow W$  is defined for all $u\in U$ by $\psi\odot \varphi (u)=\bigcup_{v\in \varphi(u)}\psi(v)\subseteq W$.
A transition $f:U\rightsquigarrow V$ is said to be \emph{hyper-deterministic}%
\footnote{In our previous texts, these transitions  were called ``quasi-deterministic", but the expression ``hyper-deterministic" is more coherent with the constraint order defined below.}
 if $card(f(u))\leq 1$ for all $u\in U$. In this case, it is often  considered  as a partial function, and we denote it by writing $f:U \dashrightarrow V$. In particular, if $card(f(u))=1$ for all $u\in U$, it is said to be \emph{deterministic} and it is considered and denoted as a total function  $f:U \rightarrow V$.

\subsubsection*{The $2$-categories $\mathbf{Tran}$ and $\mathbf{ParF}$}\label{subsubs* 2-categoris}

We  denote by $\mathbf{Tran}$\label{notation Tran}  the $2$-category that has sets as objects, transitions as arrows with the composition defined above, and such that for each couple of sets $(U,V)$ the category $\mathbf{Tran}(U,V)$ is  given by ordering the set of  {transitions} $U\rightsquigarrow V$  by  \emph{the constraint order}  defined  for all $\varphi, \psi \in \mathbf{Tran}(U,V)$
by
\[\varphi\leq\psi\Leftrightarrow \varphi\supseteq \psi\]
where $\varphi\supseteq \psi$ means that for all $u\in U$, $\varphi(u)\supseteq\psi(u)$. If $\varphi\leq\psi$, we say that $\psi$ is more \emph{constraining} than $\varphi$, or that $\varphi$ is \emph{laxer} than  $\psi$. Thus, there exists a $2$-cell $\varphi \Rightarrow\psi$ if and only if $\psi$ is more constraining than  $\varphi$.

We'll denote by $\mathbf{ParF}$\label{notation ParF} the sub-$2$-category of $\mathbf{Tran}$ obtained by keeping all sets as objects and, as $1$-cells, only the hyper-deterministic transitions between them, that is partial functions. Thus we have these inclusions of $2$-categories:
\[
\mathbf{Sets}\subseteq\mathbf{ParF}\subseteq\mathbf{Tran}.
\]

Given any small category $\mathbf{D}$, we write $\alpha:\mathbf{D}\rightharpoondown\mathbf{Tran}$ to indicate that $\alpha$ is a lax-functor from the discrete $2$-category $\mathbf{D}$  to $\mathbf{Tran}$.
Instead of $\alpha(S) \stackrel{\alpha(d)}{\rightsquigarrow}\alpha(T)$, the image of a $\mathbf{D}$-arrow $S\stackrel{d}{\rightarrow}T$ by $\alpha$ is denoted by $S^\alpha\stackrel{d^\alpha}{\rightsquigarrow}T^\alpha$ .

\begin{rmq}\label{rmq terminology TRAN instead of P} Of course, as a category, $\mathbf{Tran}$ coincides with $\mathbf{Rel}$, the category of sets with binary relations as arrows, but we prefer to emphasize the \emph{transition} point of view with this notation. In \cite{Dugowson:201112}, \cite{Dugowson:201203}, \cite{Dugowson:20150809} and \cite{Dugowson:20160831}, it was denoted $\mathbf{P}$ (for ``possible").
\end{rmq}

\begin{rmq}\label{rmq A nat isn't Nat}
Given $\alpha$ and $\beta$ two lax-functors from $\mathbf{D}$ to $\mathbf{Tran}$, we have to distinguish between 
the set --- denoted $\mathrm{nat}_{\mathbf{D}}(\alpha,\beta)$ or $\mathrm{nat}(\alpha,\beta)$ --- of all families of transitions  $(S^\alpha \stackrel{\delta_S}{\rightsquigarrow} S^\beta)_{S\in\dot{\mathbf{D}}}$ such that
\[\forall (S\stackrel{d}{\rightarrow}T)\in\overrightarrow{\mathbf{D}},
\delta_T\odot d^\alpha\subseteq d^\beta\odot \delta_S,
\]
and
the set $\mathrm{Nat}(\alpha,\beta)$\label{notation Nat} of lax-natural transformations from $\alpha$ to $\beta$.
Indeed, such a lax-natural transformation --- denoted by $\delta:\alpha\looparrowright\beta$ --- is defined not only by the data of the associated family ${(\delta_S)}_{S\in\dot{\mathbf{D}}}\in \mathrm{nat}(\alpha,\beta)$, but also by 
its domain $\alpha$, 
and its codomain $\beta$. To underline this nuance, we sometimes write 
\[\mathrm{Nat}(\alpha,\beta)=
\{(\alpha,\delta,\beta),\,
\delta\in \mathrm{nat}(\alpha,\beta)
\}.\]
For example, note that, if  $\alpha_1$, $\beta_1$, $\alpha_2$ and $\beta_2$ are some lax-functors $\mathbf{D}\rightharpoondown\mathbf{Tran}$ such that $(\alpha_1,\beta_1)\neq(\alpha_2,\beta_2)$ but that, for all $S\in\dot{\mathbf{D}}$, $S^{\alpha_1}=S^{\alpha_2}$ and $S^{\beta_1}=S^{\beta_2}$ then, because domains or codomains differ,
$\mathrm{Nat}(\alpha_1,\beta_1)\cap \mathrm{Nat}(\alpha_2,\beta_2)=\emptyset$, while we can have, and we will often have,
$\mathrm{nat}(\alpha_1,\beta_1)\cap \mathrm{nat}(\alpha_2,\beta_2)\neq\emptyset$.
\end{rmq}

\subsubsection*{Some $2$-categories of sets with $L$-families of transitions as arrows}\label{subsection TrL}

For any non-empty set $L$ we define a $2$-category denoted by 
$\mathbf{Tran}^{\underrightarrow{\scriptstyle{L}}}$ taking sets as $0$-cells  and, for each couple of sets $(U,V)$, the category $\mathbf{Tran}^{\underrightarrow{\scriptstyle{L}}}(U,V)$ being defined by 
\[
\mathbf{Tran}^{\underrightarrow{\scriptstyle{L}}}(U,V) =
(\mathbf{Tran}(U,V))^L.
\]
In other words, for a given domain $U$ and a given codomain $V$, a $1$-cell $\varphi$  in $\mathbf{Tran}^{\underrightarrow{\scriptstyle{L}}}$ is an $L$-family $(\varphi_\lambda)_{\lambda\in L}$ of transitions $\varphi_\lambda:U\rightsquigarrow V$. We  sometimes write 
$\varphi: U \twosquig_L  V $ or $ U {\stackrel{\varphi}{\twosquig}}_L V $
to indicate that $\varphi$ is such a family.

The composition of $1$-cells is naturally defined by 
\[\varphi\odot\psi
=(\varphi_\lambda)_{\lambda\in L}\odot (\psi_\lambda)_{\lambda\in L}
=(\varphi_\lambda\odot\psi_\lambda)_{\lambda\in L},\] 
and there is a $2$-cell $\varphi\Rightarrow\psi$ if and only if  $\varphi\leq \psi$, that is $\varphi_\lambda\supseteq \psi_\lambda$ for all 
$\lambda\in L$.

Similarly, we denote by $\mathbf{ParF}^{\underrightarrow{\scriptstyle{L}}}$  the sub-$2$-category of $\mathbf{Tran}^{\underrightarrow{\scriptstyle{L}}}$ obtained by keeping  sets as objects and, as $1$-cells, only the $L$-families of hyper-deterministic transitions between them, that is $L$-families of partial functions, 
and by $\mathbf{Sets}^{\underrightarrow{\scriptstyle{L}}}$  
the category of sets and, as arrows, $L$-families of total functions, so we have
\[
\mathbf{Sets}^{\underrightarrow{\scriptstyle{L}}}
\subseteq\mathbf{ParF}^{\underrightarrow{\scriptstyle{L}}}
\subseteq\mathbf{Tran}^{\underrightarrow{\scriptstyle{L}}}.
\]

As in the case when $L$ is a singleton, we  write 
$\alpha:\mathbf{D}\rightharpoondown\mathbf{Tran}^{\underrightarrow{\scriptstyle{L}}}$ 
to indicate that $\alpha$ is a lax-functor from the discrete $2$-category $\mathbf{D}$  to $\mathbf{Tran}^{\underrightarrow{\scriptstyle{L}}}$ and  the image of a $\mathbf{D}$-arrow $S\stackrel{d}{\rightarrow}T$ by such an $\alpha$ is 
 denoted by $S^\alpha{\stackrel{d^\alpha}{\twosquig}}_L T^\alpha$ 
%(or by $S^\alpha\stackrel{(d^\alpha_\lambda)_{\lambda\in L}}{\twosquig}T^\alpha$ )
 instead of $\alpha(S) {\stackrel{\alpha(d)}{{\twosquig}_L}}\alpha(T)$.

\section{Open dynamics}

\subsection{Multi-dynamics}

\subsubsection{$L$-dynamics on a category $\mathbf{D}$}\label{subsubs LdynD}

Let $L$ be a non-empty set, and $\mathbf{D}$  a small category.
A  lax-functor 
$\alpha:\mathbf{D}\rightharpoondown\mathbf{Tran}^{\underrightarrow{\scriptstyle{L}}}$
is said to be \emph{disjunctive} if for all objects $S\neq T$ in $\mathbf{D}$, we have
$S^\alpha \cap T^\alpha=\emptyset$.

\begin{df}[$L$-dynamics on $\mathbf{D}$] A \emph{multi-dynamic $\alpha$ on $\mathbf{D}$ with $L$ as set of parameter values}, or simply an
 \emph{$L$-dynamic 
on $\mathbf{D}$},  is a disjunctive lax-functor $\alpha:\mathbf{D}\rightharpoondown\mathbf{Tran}^{\underrightarrow{\scriptstyle{L}}}$.
\end{df}

For each $S\in\dot{\mathbf{D}}$, the elements of the set $S^\alpha$ are called the \emph{states of $\alpha$ of type $S$}, and we denote by $st(\alpha)$ the set $\bigsqcup_{S\in\dot{\mathbf{D}}}{S^\alpha}$ of all states of $\alpha$.
The category  $\mathbf{D}$ is called the \emph{engine} of $\alpha$, its arrows $(S\stackrel{d}{\rightarrow}T)\in\overrightarrow{\mathbf{D}}$ are called \emph{durations}. 

By definition of lax-functors between bicategories, 
an $L$-multi-dynamic $\alpha$ associates with each duration $(S\stackrel{d}{\rightarrow}T)\in\overrightarrow{\mathbf{D}}$ an $L$-family  of transitions $
d^\alpha=(d^\alpha_\lambda)_{\lambda\in L}:
S^\alpha \,
{\twosquig}_L 
T^\alpha$ 
such that, for each $\lambda\in L$ and any  composable arrows $R {\stackrel{d}{\longrightarrow}} S  {\stackrel{e}{\longrightarrow}} T$, we have
\begin{itemize}
\item (disjunctivity) $S\neq T \Rightarrow S^\alpha \cap T^\alpha=\emptyset$,
\item (lax identity) $(Id_S)^\alpha_\lambda \subseteq Id_{S^\alpha}$,
\item (lax composition) $(e\circ d)^\alpha_\lambda \subseteq e^\alpha_\lambda \odot d^\alpha_\lambda$.
\end{itemize}
 A state $u\in S^\alpha$ such that $(Id_S)^\alpha_\lambda(u)=\emptyset$ is said to be \emph{offside for the parameter value $\lambda\in L$}, and it is simply said to be \emph{offside} if it is offside for all parameter values. A state that is not offside is said to be \emph{onside}.
If the lax-functor $\alpha$ is in fact a functor $\mathbf{D} \rightarrow\mathbf{Tran}^{\underrightarrow{\scriptstyle{L}}}$, we  say that the multi-dynamic $\alpha$ is \emph{functorial} or \emph{strict}. 
An $L$-dynamic on $\mathbf{D}$ is said to be \emph{deterministic} (resp. \emph{hyper-deterministic}) if for each duration $d\in\overrightarrow{\mathbf{D}}$ and each parameter value $\lambda\in L$, the transition $d^\alpha_\lambda$ is deterministic (resp. hyper-deterministic). In other words, a deterministic $L$-dynamic on $\mathbf{D}$ is a disjunctive functor\footnote{Obviously, an $L$-dynamic that is deterministic is necessarily functorial.} $\mathbf{D}\rightarrow \mathbf{Sets}^{\underrightarrow{\scriptstyle{L}}}$, and a hyper-deterministic $L$-dynamic on $\mathbf{D}$ is a disjunctive lax-functor $\mathbf{D}\rightharpoondown\mathbf{ParF}^{\underrightarrow{\scriptstyle{L}}}$.

\begin{rmq} \label{rmq terminology multi-dynamiques} In \cite{Dugowson:20150809}, multi-dynamics were called \emph{multi-dynamiques sous-catégoriques} --- and \emph{multi-dynamiques catégoriques} in the functorial case --- whereas they were called \emph{multi-dynamiques sous-fonctorielles} in \cite{Dugowson:20160831}. 
\end{rmq}

\subsubsection{The category $\mathbf{MonoDyn}_\mathbf{D}$ of mono-dynamics on $\mathbf{D}$}\label{subsubs mono-dynamics}

In the particular case where $L$ is a singleton $\{*\}$, an $L$-dynamic is called a \emph{mono-dynamic} (or simply a \emph{dynamic}) on $\mathbf{D}$. 
Taking lax-natural transformations between mono-dynamics on $\mathbf{D}$ as morphisms, we obtain the category\footnote{\label{fnote terminology monodynamics}In \cite{Dugowson:20150809}, mono-dynamics were called \emph{dynamiques sous-catégoriques} --- and \emph{dynamiques catégoriques} in the functorial case --- and the category $\mathbf{MonoDyn}_\mathbf{D}$ was denoted by $\mathbf{DySC_{(D)}}$. In \cite{Dugowson:20160831}, they were called \emph{mono-dynamiques sous-fonctorielles}.}
$\mathbf{MonoDyn}_\mathbf{D}$ of mono-dynamics on $\mathbf{D}$. These morphisms are called \emph{dynamorphisms}, and we  write $\delta:\alpha\looparrowright\beta$ to indicate that $\delta$ is a dynamorphism from $\alpha$ to $\beta$. Such
a dynamorphism $\delta\in \mathbf{MonoDyn}_\mathbf{D}(\alpha,\beta)$ is said to be \emph{deterministic} (resp. \emph{hyper-deterministic}) iff all transitions $\delta_S$ are deterministic (resp. hyper-deterministic).

\begin{rmq}\label{rmq naturality}
Following the remark \ref{rmq A nat isn't Nat}, we have to distinguish between $\mathrm{nat}(\alpha,\beta)$ and $\mathbf{MonoDyn}_\mathbf{D}(\alpha,\beta)=\mathrm{Nat}(\alpha,\beta)$. Nevertheless, as long as there is no ambiguity, we shall denote as usual by a same letter  a lax-natural transformation $\delta$ and the corresponding family $\delta=(\delta_S)_{S\in\dot{\mathbf{D}}}$ of transitions. Also note that disjunctivity of $\alpha$ implies that a family of transitions $\left(\delta_S:S^\alpha\rightsquigarrow S^\beta\right)_{S\in\dot{\mathbf{D}}}$ can be seen as a single transition 
$\delta:
st(\alpha)\rightsquigarrow st(\beta)$ 
with, for each $S\in\dot{\mathbf{D}}$ and each $u\in S^\alpha$, 
$\delta(u)=\delta_S(u)\subseteq S^\beta\subseteq st(\beta)$. 
Then, a dynamorphism $\delta:\alpha\looparrowright\beta$ 
is often seen as such a transition 
$st(\alpha)\rightsquigarrow st(\beta)$.
\end{rmq}

\subsubsection{Clocks on $\mathbf{D}$}\label{subsubs clocks}

\begin{df}
A monodynamic $\mathbf{h}$ on $\mathbf{D}$ that is deterministic is called a \emph{clock} on $\mathbf{D}$. Its states are called $\mathbf{h}-$\emph{instants} (or simply \emph{instants}).
\end{df}

Thus, a clock on $\mathbf{D}$ is nothing but a disjunctive functor $\mathbf{D}\rightarrow \mathbf{Sets}$. A pre-order relation, called \emph{anteriority} and denoted by $\leq_\mathbf{h}$, is defined on  $st(\mathbf{h})$  by 
\[
(s\leq_\mathbf{h} t )
\Leftrightarrow
(\exists e\in \overrightarrow{\mathbf{D}}, e^\mathbf{h}(s)=t)\] for all instants $s$ and $t$. 
We define the category $\mathbf{Clocks_D}$ of clocks on $\mathbf{D}$ taking deterministic dynamorphisms as morphisms between them. It is equivalent to the topos of presheaves on $\mathbf{D}^{op}$.

\subsubsection{The category $L-\mathbf{Dyn}_\mathbf{D}$ of $L$-dynamics on $\mathbf{D}$}

 We denote by $L-\mathbf{Dyn}_\mathbf{D}$ the category whose objects are $L$-dynamics on $\mathbf{D}$, and with  arrows $\delta:\alpha\looparrowright\beta$ --- called \emph{$({\mathbf{D}, L})$-dynamorphisms} --- given by the families of transitions 
 $(S^\alpha \stackrel{\delta_S}{\rightsquigarrow} S^\beta)_{S\in\dot{\mathbf{D}}}$
  that are lax-natural from the mono-dynamic $\alpha_\lambda$ to the mono-dynamic $\beta_\lambda$ for all $\lambda\in L$, that is such that
\[\forall\lambda\in L, \forall (S\stackrel{d}{\rightarrow}T)\in\overrightarrow{\mathbf{D}},
\delta_T\odot d^\alpha_\lambda\subseteq d^\beta_\lambda\odot \delta_S.
\]
Following the remark \ref{rmq A nat isn't Nat},  we can formulate this by writing
\[
L-\mathbf{Dyn}_\mathbf{D}(\alpha,\beta)=
\{
(\alpha,\delta,\beta), \delta\in \bigcap_{\lambda\in L} \mathrm{nat}(\alpha_\lambda,\beta_\lambda)
\}\]
or even, with the usual omission of domain and codomain when there is no ambiguity, by
$
L-\mathbf{Dyn}_\mathbf{D}(\alpha,\beta)=\bigcap_{\lambda\in L} \mathrm{nat}(\alpha_\lambda,\beta_\lambda)$.

\subsubsection{The category $\mathbf{MultiDyn}_\mathbf{D}$ of multi-dynamics on $\mathbf{D}$}

Let $L$ and $M$ be some non-empty sets, and  $\alpha:\mathbf{D}\rightharpoondown 
\mathbf{Tran}^{\underrightarrow{\scriptstyle{L}}}$ 
and 
$\beta:\mathbf{D}\rightharpoondown 
\mathbf{Tran}^{\underrightarrow{\scriptstyle{M}}}$
be  multi-dynamics on $\mathbf{D}$. 
\begin{df}
A \emph{$\mathbf{D}$-dynamorphism} $\alpha\looparrowright\beta$ 
is a couple $(\theta,\delta)$ 
with $\theta:L\rightarrow M$ a map, 
and\footnote{Using the notation explained in the remark \ref{rmq A nat isn't Nat}.}
 $\delta\in\bigcap_{\lambda\in L} 
\mathrm{nat}_\mathbf{D}(\alpha_\lambda,\beta_{\theta(\lambda)})$.
\end{df}

 Thus, to be a dynamorphism, $(\theta,\delta)$ must satisfy the lax-naturality condition
\[\forall \lambda\in L,\forall (S\stackrel{d}{\rightarrow}T)\in\overrightarrow{\mathbf{D}}, \delta_T\odot d^\alpha_\lambda\subseteq d^\beta_{\theta(\lambda)}\odot \delta_S.\]

We then obtain the category $\mathbf{MultiDyn}_{\mathbf{D}}$ taking as objects all  multi-dynamics on $\mathbf{D}$, and as arrows all $\mathbf{D}$-dynamorphisms between them. Naturally, such  a dynamorphism $(\theta,\delta)$ is said to be \emph{(hyper-)deterministic} if for every object $S\in\dot{\mathbf{D}}$, $\delta_S$ is (hyper-)deterministic.

\begin{rmq} For any set $L$ with $card(L)\geq 2$, $L-\mathbf{Dyn}_\mathbf{D}$ is a non-full subcategory of $\mathbf{MultiDyn}_{\mathbf{D}}$, whereas  
$\mathbf{MonoDyn}_\mathbf{D}$ is a full one.  We can in particular consider dynamorphisms between mono-dynamics on $\mathbf{D}$ and multi-dynamics on $\mathbf{D}$. For example, if  $L$ is a non-empty set, $\alpha$ an $L$-dynamic on   $\mathbf{D}$, and $\mathbf{h}$ a clock on the same engine, then a dynamorphism $\mathfrak{s}:\mathbf{h}\looparrowright \alpha$ is a couple
$\mathfrak{s}=(\lambda,\sigma)$ with $\lambda\in L$ and $\sigma=(S^\mathbf{h} \stackrel{\sigma_S}{\rightsquigarrow} S^\alpha)_{S\in\dot{\mathbf{D}}}$ such that 
\[\forall (S\stackrel{d}{\rightarrow}T)\in\overrightarrow{\mathbf{D}},
\sigma_T\odot d^\mathbf{h}\subseteq d^\alpha_\lambda\odot \sigma_S, 
\] whereas a dynamorphism $\tau:\alpha\looparrowright \mathbf{h}$ is a family of transitions $\tau=(S^\alpha \stackrel{\tau_S}{\rightsquigarrow} S^\mathbf{h})_{S\in\dot{\mathbf{D}}}$ such that
\[\forall (S\stackrel{d}{\rightarrow}T)\in\overrightarrow{\mathbf{D}}, \forall \lambda\in L,
\tau_T\odot d^\alpha_\lambda\subseteq d^\mathbf{h}\odot \tau_S.
\]
\end{rmq}

\subsubsection{The category $\mathbf{MultiDyn}$ of multi-dynamics}

Let $\alpha:\mathbf{D}\rightharpoondown 
\mathbf{Tran}^{\underrightarrow{\scriptstyle{L}}}$ 
and 
$\beta:\mathbf{E}\rightharpoondown 
\mathbf{Tran}^{\underrightarrow{\scriptstyle{M}}}$
be  multi-dynamics with possibly different sets of parameter values and different engines.

\begin{df} A \emph{dynamorphism} $\alpha\looparrowright\beta$  consists, in addition to the data of $\alpha$ and $\beta$, of that of a triple $(\theta,\Delta,\delta)$ 
with 
\begin{itemize}
\item $\theta:L\rightarrow M$ a map, 
\item $\Delta:\mathbf{D}\rightarrow\mathbf{E}$ a functor,
\item $\delta\in\bigcap_{\lambda\in L}\mathrm{nat}_\mathbf{D}(\alpha_\lambda,\beta_{\theta(\lambda)}\circ\Delta)$.
\end{itemize}
\end{df}

The last condition means that the lax-naturality condition
\[\forall \lambda\in L,\forall (S\stackrel{d}{\rightarrow}T)\in\overrightarrow{\mathbf{D}}, \delta_T\odot d^\alpha_\lambda\subseteq (\Delta d)^\beta_{\theta(\lambda)}\odot \delta_S\]
has to be satisfied. 
The category $\mathbf{MultiDyn}$ of multi-dynamics is then defined taking as objects all multi-dynamics, and as arrows all dynamorphisms between them. The full subcategory of  $\mathbf{MultiDyn}$ obtained taking mono-dynamics (resp. clocks) as objects is denoted by $\mathbf{MonoDyn}$ (resp. $\mathbf{Clocks}$). In general, $\mathbf{MonoDyn}_\mathbf{D}$ (resp.   $\mathbf{Clocks}_\mathbf{D}$) is a non-full subcategory of $\mathbf{MonoDyn}$ (resp.   $\mathbf{Clocks}$).

\subsection{Open dynamics: definition, realizations, quotients}

\subsubsection{Definition of open dynamics}

\begin{df}\label{df open dynamics}  An \emph{open dynamic $A$ with engine $\mathbf{D}$} is the data
\[A=\left((\alpha:\mathbf{D}\rightharpoondown \mathbf{Tran}^{\underrightarrow{\scriptstyle{L}}}) \stackrel{\rho}{\looparrowright}  \mathbf{h}\right)\] of
\begin{itemize}
\item a non-empty set $L$ of parameter values,
\item an $L$-dynamic $\alpha\in L-\mathbf{Dyn}_\mathbf{D}$,
\item a clock $(\mathbf{h}:\mathbf{D}\rightarrow \mathbf{Sets})\in \mathbf{Clocks}_{\mathbf{D}}$,
\item a \emph{deterministic} dynamorphism 
$
{\rho}\in \mathbf{MultiDyn}_{\mathbf{D}}(\alpha,\mathbf{h})
$ called \emph{datation}.
\end{itemize}
\end{df}

An open dynamic with engine $\mathbf{D}$ is also called an open dynamic \emph{on $\mathbf{D}$}.
An open dynamic is said to be \emph{intemporal} if its engine is the terminal category $\mathbf{1}$. The states of $\alpha$ are also called the states of $A$, thus we set: $st(A)=st(\alpha)$.
If the parametric set $L$ is a singleton, $A$ is said to be an \emph{open mono-dynamic} or, sometimes, an \emph{opaque dynamic}.

\begin{rmq}
For each $\lambda\in L$, we  naturally denote by $A_\lambda$  the open mono-dynamic obtained by restricting parametric values to $\lambda$, that is 
\[A_\lambda=\left(
(\alpha_\lambda:\mathbf{D}\rightharpoondown \mathbf{Tran})
\stackrel{\rho}{\looparrowright}
(\mathbf{h}:\mathbf{D}\rightarrow \mathbf{Sets})
\right).\]
\end{rmq}

According to the definitions given in \textbf{§\,2.4.2} of \cite{Dugowson:20150809} and \textbf{§\,1.2.2} of \cite{Dugowson:20160831}, a \emph{dynamorphism} from an open dynamic
\[A=\left((\alpha:\mathbf{D}\rightharpoondown \mathbf{Tran}^{\underrightarrow{\scriptstyle{L}}}) \stackrel{\rho}{\looparrowright}  (\mathbf{h}:\mathbf{D}\rightarrow \mathbf{Sets})\right)\]
to an open dynamic
\[B=\left((\beta:\mathbf{E}\rightharpoondown \mathbf{Tran}^{\underrightarrow{\scriptstyle{M}}}) \stackrel{\tau}{\looparrowright}  (\mathbf{k}:\mathbf{E}\rightarrow \mathbf{Sets})\right)\]
is a quadruplet $(\theta, \Delta, \delta, \varepsilon)$ with
\begin{itemize}
\item $(\theta,\Delta,\delta)\in\mathbf{MultiDyn}(\alpha,\beta)$,
\item $(\Delta,\varepsilon)\in\mathbf{MonoDyn}(\mathbf{h},\mathbf{k})$,
\item this lax synchronization condition satisfied:
\[
\forall S\in\dot{\mathbf{D}}, 
\tau_{\Delta_S}\odot\delta_S\subseteq\varepsilon_S\odot\rho_S.
\]
\end{itemize}

We  denote by $\mathbf{ODyn}$ the category of all open dynamics, with dynamorphisms as arrows.

\subsubsection{Realizations of an open dynamic}
\label{subsubs realizations of an open dynamic}

%With the definitions given above, the other notions of our systemic theory of interactivity keep the same definitions as previously given in \cite{Dugowson:20160831}: parametric quotients, interaction relations, connectivity structures of an interaction relation, interactive families, open dynamics generated by an interactive family, ...

Let $A=\left((\alpha:\mathbf{D}\rightharpoondown \mathbf{Tran}^{\underrightarrow{\scriptstyle{L}}}) \stackrel{\rho}{\looparrowright}  (\mathbf{h}:\mathbf{D}\rightarrow \mathbf{Sets})\right)$ be an open dynamic.

\begin{df}
A \emph{realization} (or a \emph{solution}) of $A$ is a hyper-deterministic dynamorphism $(\mathfrak{s}:\mathbf{h}\looparrowright\alpha)\in \mathbf{MultiDyn}_\mathbf{D}(\mathbf{h},\alpha)$ such that  the lax condition \[\rho\odot \mathfrak{s}\subseteq Id_\mathbf{h}\] be satisfied.
\end{df}

In other words\footnote{See \cite{Dugowson:20160831}, \textbf{§\,1.3.1.}},  a realization of $A$ is a couple $\mathfrak{s}=(\lambda,\sigma)$ with $\lambda\in L$ and  $\sigma:st(\mathbf{h})\dashrightarrow st(\alpha)$ a partial function defined on a subset $\mathrm{Def}_\sigma\subseteq st(\mathbf{h})$ such that:

\begin{enumerate}
\item $\forall t\in \mathrm{Def}_\sigma, \rho(\sigma(t))=t$,
\item $\forall S\in\dot{\mathbf{D}},\forall t\in S^\mathbf{h}\cap  \mathrm{Def}_\sigma, \sigma(t)\in S^\alpha$,
\item $\forall (S\stackrel{d}{\rightarrow}T)\in\overrightarrow{\mathbf{D}}$, 
 $\forall t\in S^\mathbf{h}$, 
 \[d^\mathbf{h}(t)\in \mathrm{Def}_\sigma\Rightarrow 
\left[ t\in \mathrm{Def}_\sigma 
\,\,\mathrm{and}\,\,
\sigma(d^\mathbf{h}(t))\in d^\alpha_\lambda (\sigma(t))\right].
\]
\end{enumerate}

The set of realizations of $A$ is denoted by $\mathfrak{S}_A$.
Given $\mathfrak{s}=(\lambda,\sigma)\in \mathfrak{S}_A$, we  call $\lambda$ the \emph{parametric part} or the \emph{incoming part} of this realization, $\sigma$ its \emph{outgoing part}, and we set 
\[
\mathrm{In}(\mathfrak{s}):=\lambda
\quad\mathrm{and}\quad
\mathrm{Out}(\mathfrak{s}):=\sigma.
\]

Outgoing parts of realizations of $A$ is often  called \emph{outgoing realizations of $A$} --- or even simply \emph{realizations}, if there is no ambiguity --- and their set is denoted by ${Z}_A$. 
 Thus, we have%
\footnote{For any $\lambda\in L$, the set of realizations of the open (mono) dynamic $A_\lambda$ is simply given by $\mathfrak{S}_{A_\lambda}=\{\lambda\}\times Z_{A_\lambda}$, so this latter set ${Z}_{A_\lambda}$ of outgoing parts of realizations of $A_\lambda$ is often simply called the set of its realizations.}
\[{Z}_A=\bigcup_{\lambda\in L}{Z}_{A_\lambda}.\]

A realization of $A$ is said to be empty if its outgoing part is the empty function $st(\mathbf{h})\supset \emptyset \hookrightarrow st(\alpha)$. This empty function is denoted by $\underline{\emptyset}_A$, or simply 
$\underline{\emptyset}$, if there is no ambiguity. We always have    ${Z}_A\ni\underline{\emptyset}_A$, and we  denote by $Z_A^*$ the set of non-empty outgoing realizations of $A$:
\[{Z}_A^*={Z}_A\setminus \{\underline{\emptyset}_A\}.\]

An open dynamic $A$ is said to be \emph{efficient} if the set ${Z}_A^*$ is non empty.

\paragraph{Realizations passing through a state.}
\label{subsubs realizations passing through a state}

\begin{df}
Given an open dynamic $A$, we say that  a realization $\mathfrak{s}=(\lambda,\sigma)$ of $A$ \emph{passes through} a state $a\in st(A)$ --- or equivalently that the outgoing part $\sigma$ of  $\mathfrak{s}$ passes through $a$ --- and we  write
\[\mathfrak{s}\rhd a\quad (\mathrm{or, equivalently:\sigma\rhd a}) \] 
if $\sigma(\rho(a))=a$.
\end{df}

More generally, if $E$ is a set of states of $A$, we  write
\[\mathfrak{s} \rhd E\quad (\mathrm{or, equivalently:}\,\sigma \rhd E)\]
to say that $\sigma$ passes through every $a\in E$. If $E$ is a finite set $E=\{a_1,..., a_n\}$, we can also write 
\[\sigma\rhd a_1,..., a_n.\]

\subsubsection{Parametric quotients}
\label{subsubs parametric quotient}

Let $\alpha:\mathbf{D}\rightharpoondown\mathbf{Tran}^{\underrightarrow{\scriptstyle{L}}}$  be a multi-dynamic with engine $\mathbf{D}$ and parametric set $L$.

\begin{prop}\label{prop quotient param}
If $\sim$ is an equivalence relation on $L$, and $M=L/{\sim}$ is the quotient set of $L$ by $\sim$, then the relation 
\[\forall \mu\in M, \beta_\mu=\bigcup_{\lambda\in\mu}\alpha_\lambda,\]
that is
\begin{itemize}
\item $\forall S\in\dot{\mathbf{D}}$,  $S^\beta=S^\alpha$,
\item  $\forall(e:S\rightarrow T)\in\overrightarrow{\mathbf{D}}, \forall a\in S^\beta, e^\beta_\mu(a)=\bigcup_{\lambda\in\mu}e^\alpha_\lambda(a)$,
\end{itemize}
defines a multi-dynamic $\beta$ on $\mathbf{D}$ with parametric set $M$.
\end{prop}

\begin{proof}
%\footnote{See also  prop 2, section \textbf{§\,2.1.7} in \cite{Dugowson:20150809}.}.
For every $\mu\in L/{\sim}$, and each $S\in\dot{\mathbf{D}}$, we have\footnote{Where the order relation $\leq$  is of course the constraint order 
$\varphi\leq\psi\Leftrightarrow \varphi\supseteq \psi$.}
\[
(Id_S)^{\beta}_\mu=
\bigcup_{\lambda\in\mu}{(Id_S)^\alpha_\lambda}
\subseteq
\bigcup_{\lambda\in\mu}{Id_{S^\alpha}}=Id_{S^\beta},
\] that is $(Id_S)^{\beta}_\mu\geq Id_{S^\beta}$.

Furthermore,  for each couple of composable arrows 
$R {\stackrel{f}{\longrightarrow}} S  {\stackrel{g}{\longrightarrow}} T$ in $\mathbf{D}$, we have
\[
(g \circ f)^{\beta}_\mu
=
\bigcup_{\lambda\in\mu}{(g \circ f)}^{\alpha}_\lambda
\geq
\bigcup_{\lambda\in\mu}{\left( g^\alpha_\lambda \odot f^\alpha_\lambda\right)},\]
but for each $\lambda\in\mu$, we have $f^\alpha_\lambda \subseteq f^\beta_\mu$, and the same for $g$, and then 
\[{g^\alpha_\lambda \odot f^\alpha_\lambda}
\subseteq
 g^{\beta}_\mu \odot f^{\beta}_\mu,\] 
 so 
\[
(g \circ f)^{\beta}_\mu
\subseteq
g^{\beta}_\mu \odot f^{\beta}_\mu,
\] that is
\[
(g \circ f)^{\beta}_\mu
\geq
g^{\beta}_\mu \odot f^{\beta}_\mu.
\]
\end{proof}

\begin{df}[Parametric quotient of a dynamic]\label{df reduc param}
The multi-dynamic $\beta:\mathbf{D}\rightharpoondown\mathbf{Tran}^{\underrightarrow{\scriptstyle{M}}}$ defined in proposition \ref{prop quotient param} by
\[\forall \mu\in M, \beta_\mu=\bigcup_{\lambda\in\mu}\alpha_\lambda,\]
is called the \emph{parametric quotient of $\alpha$ by $\sim$} and is denoted by  $\beta=\alpha/{\sim}$. \\
For any open dynamic
 \[A=\left((\alpha:\mathbf{D}\rightharpoondown \mathbf{Tran}^{\underrightarrow{\scriptstyle{L}}}) \stackrel{\rho}{\looparrowright}  (\mathbf{h}:\mathbf{D}\rightarrow \mathbf{Sets})\right)\] and any equivalence relation $\sim$ on $L$, we define in the same way the \emph{quotient open dynamic} $B=A/\sim$ setting
\[{B}=\left(({(\alpha/\sim)}:\mathbf{D}\rightharpoondown \mathbf{Tran}^{\underrightarrow{(\scriptstyle{L}/\sim)}}) \stackrel{\tilde{\rho}}{\looparrowright}  (\mathbf{h}:\mathbf{D}\rightarrow \mathbf{Sets})\right)
\]
where, for every $b\in S^{\alpha/\sim}=S^\alpha$, $\tilde{\rho}(b)=\rho(b)$.
\end{df}

\subsection{Examples of open dynamics}

%\subsubsection{Bushaw's dynamics}

\begin{example}[Bushaw's dynamics]\label{exm Bushaw}
In her 1965 article \cite{Bastiani_Ehresmann:196505}, Andrée Bastiani (-Ehresmann) cited  Donald W. Bushaw's 1963 article \cite{Bushaw:19630423} in which this one introduced some continu\-ous \emph{dynamical  polysystems} that correspond --- leaving aside topological aspects  --- to our deterministic open dynamics with the group $(\mathbf{R},+)$ as engine and with clock the real \emph{existential clock}%%%
\footnote{About the existential clock of a category, see \cite{Dugowson:201012}.}%%%
\, $\xi=\xi_{(\mathbf{R},+)}$ (defined by $st(\xi)=\mathbf{R}$ and $d^\xi(t)=t+d$ for all reals $t$ and $d$):
\[(\alpha:(\mathbf{R},+) \rightarrow \mathbf{Sets}^{\underrightarrow{\scriptstyle{L}}}) \stackrel{\rho}{\looparrowright}  (\xi : (\mathbf{R},+)\rightarrow \mathbf{Sets})\]
\noindent such that the following additional condition (``non-anticipation'') be satisfied: for all $\lambda_1, {\lambda_2} \in L$ and  $t_0\in\mathbf{R}$ there exists a unique $\lambda\in L$ such that, for all states $s\in st(\alpha)$ with $\tau(s)=t_0$, we have 
\begin{itemize}
\item $\forall d\in\mathbf{R}_{-}, d^\alpha_\lambda(s)= d^\alpha_{\lambda_1}(s)$,
\item $\forall d\in\mathbf{R}_{+}, d^\alpha_\lambda(s)= d^\alpha_{{\lambda_2}}(s)$.
\end{itemize}

Thus, with each Bushaw's dynamical polysystem is canonically associated a deterministic open multi-dynamic on $\mathbf{R}$. Reciprocally, by choosing convenient topological structures on the set of states and on the set of parameter values, some Bushaw's dynamical polysystem(s) can be associated with each deterministic open multi-dynamic on $\mathbf{R}$ endowed with the existential clock $\xi$ and satisfying the ``non-anticipation'' property.

\paragraph{Realizations of Bushaw's dynamics.} $(\mathbf{R},+)$ being a group, every non\-empty outgoing realization of the considered deterministic open dynamic is defined on the whole real line and, with the topological assumptions of Bushaw's paper, it is necessarily continuous. For its part, Bushaw doesn't explicitly  define the realizations (solutions) of his systems. Nevertheless,
for each $\varphi\in L$, Bushaw denotes again by $\varphi$ the map $E\times \mathbf{R}\rightarrow E$, where $E=st(\alpha)$, defined with our notations by $\varphi(e,d)=d^\alpha_\varphi(e)$. Then, for each given state $e\in E$,  the map $\sigma:\mathbf{R}\ni t \mapsto \varphi(e,t-\tau(e))\in E$ constitutes the single realization of the considered deterministic dynamic such that $\sigma(\tau(e))=e$. It is defined on all $\mathbf{R}$, and it is continuous. Thus, for each $\lambda$, there is an implicitly notion of realization that coincides with ours, even if some notion of partial solution could perhaps be closer to the spirit of his work (because of the local aspect of the parameters $\lambda$).

\end{example}

\mbox{}

%\subsubsection{$\Phi$, a  deterministic intemporal mono-dynamic}

\begin{example}[$\Phi$, a  deterministic intemporal mono-dynamic]\label{exm Phi}
An intemporal dynamic is functorial if and only if it is deterministic, and in this case its behavior cannot depend on any parameter, since the image by the dynamic of the only duration $\overrightarrow{0}$ is necessarily the identity of the set of states. For example, we can consider the deterministic intemporal monodynamic $\Phi$ for which the set of states is $\{0,1\}$, that is
\[
\Phi=
\left((\phi:\mathbf{1}\rightarrow \mathbf{Sets}) \stackrel{!}{\looparrowright}  (\xi_\mathbf{1}:\mathbf{1}\rightarrow \mathbf{Sets})\right)
\]
where
\begin{itemize}
\item $st(\phi)=\bullet^\phi=\{0,1\}$, 
\item $\overrightarrow{0}^\phi=Id_{\{0,1\}}$,
\item $\xi_\mathbf{1}$ is the canonical clock\footnote{That is both the existential clock and the essential clock of $\mathbf{1}$. About the existential clock and the essential clock of a category, see \cite{Dugowson:201012}.} of $\mathbf{1}$, which has only one instant $0$, 
\end{itemize}
and $\phi \stackrel{!}{\looparrowright} \xi_\mathbf{1}$ is the necessarily constant dynamorphism.

\paragraph{Realizations of $\Phi$.} We immediately see that
\[
\mathfrak{S}_\Phi=Z_\Phi=\{\underline{\emptyset}_\Phi, 0, 1\}.
\]
\end{example}

\mbox{}
 
%\subsubsection{$\Upsilon$ and $\Upsilon_*$, two one-step (hyper-)deterministic cells}

\begin{example}[$\Upsilon$, a one-step deterministic cell]\label{exm Upsilon}

We set
\[
\Upsilon=
\left((\upsilon:\mathbf{D}_\Upsilon\rightarrow \mathbf{Tran}^{\underrightarrow{\scriptstyle{L_\Upsilon}}}) \stackrel{!}{\looparrowright}  (\zeta_{\mathbf{D}_\Upsilon}:{\mathbf{D}_\Upsilon}\rightarrow \mathbf{Sets})\right)
\] where 
\begin{itemize}
\item $\mathbf{D}_\Upsilon=(T_0 \stackrel{d}{\rightarrow} T_1)\simeq (\bullet \rightarrow \bullet)$, the category with two objects and a single non-trivial arrow between them, which we  call the \emph{one-step category},
\item $\zeta_{\mathbf{D}_\Upsilon}$ is the \emph{essential clock}\footnote{See \cite{Dugowson:201012}.} of ${\mathbf{D}_\Upsilon}$, for which the set of instants associated with each $T_k$ is a singleton, say ${T_k}^{\zeta_{\mathbf{D}_\Upsilon}}=\{t_k\}$,
\item $\forall k\in\{0,1\}$, ${(T_k)}^\upsilon=\{t_k\}\times \{0,1\}$, 
\item $L_\Upsilon=\{0,1\}^{\{0,1\}}$,
\item $\forall \lambda\in L_\Upsilon$,
 $\forall k\in\{0,1\}$, $({Id_{(T_k)}})^\upsilon_\lambda=Id_{({(T_k)}^\upsilon)}$ (since ${\Upsilon}$ is functorial), 
\item $\forall \lambda\in L_\Upsilon$, $\forall s\in\{0,1\}$, $d^\upsilon_\lambda(t_0,s)=(t_1,\lambda(s))$
%(\footnote{, that is: if $\lambda=(1,0)$, then $d^\upsilon_\lambda\equiv (1,0)$; if $\lambda=(0,1)$, then $(Id)$ : $d^\upsilon_\lambda = (1,Id_{\{0,1\}}$;  if $\lambda=(1,1)$, then $d^\upsilon_\lambda \equiv (1,1)$; if $\lambda=(1,0)$, then switch $d^\upsilon_\lambda \equiv (1,1)$})
,
\item $\upsilon \stackrel{!}{\looparrowright} \zeta_{\mathbf{D}_\Upsilon}$ is the  unique possible deterministic dynamorphism here (since there is a unique instant for each temporal type 
$T_k\in\dot{\mathbf{D}_\Upsilon}$).
\end{itemize}

\paragraph{Realizations of $\Upsilon$.} An outgoing realization  of $\Upsilon$ can be identified with some partial function $\sigma:\{t_0,t_1\}\dashrightarrow\{0,1\}$ such that $\mathrm{Def}_\sigma\in\{\emptyset, \{t_0\}, \{t_0,t_1\}\}$. With this identification, we can write $\mathfrak{S}_\Upsilon$ as the set of all couples $(\lambda,\sigma)$
with $\lambda\in{L_\Upsilon}$ and
$\sigma=\underline{\emptyset}_\Upsilon$,
or
$\sigma\in\{0,1\}^{\{t_0\}}$,
or $\sigma\in\{0,1\}^{\{t_0,t_1\}}$ with $\sigma(t_1)=\lambda(\sigma(t_0))$. Then,
\[
Z_\Upsilon=
\{\underline{\emptyset}_\Upsilon\}
\cup
\{0,1\}^{\{t_0\}}
\cup
\{0,1\}^{\{t_0,t_1\}}.
\]
\end{example}

\mbox{}

\begin{example}[${\Upsilon_*}$, a one-step  hyper-deterministic cell]\label{exm Upsilon*}

This is a functorial hyper-deterministic variant of the example \ref{exm Upsilon}, keeping the same engine $\mathbf{D}_{\Upsilon_*}=\mathbf{D}_{\Upsilon}=(T_0 \stackrel{d}{\rightarrow} T_1)$, the same states and the same clock $\zeta=\zeta_{\mathbf{D}_\Upsilon}$ but including new parameter values to permit a state to ask to ``\emph{exit the game}''. More precisely,
\[
{\Upsilon_*}=
\left(({\upsilon_*}:\mathbf{D}_{\Upsilon_*}\rightarrow \mathbf{Tran}^{\underrightarrow{\scriptstyle{L_{\Upsilon_*}}}}) \stackrel{!}{\looparrowright}  \zeta \right)
\] where
\begin{itemize}
\item $\forall k\in\{0,1\}$, ${(T_k)}^{\upsilon_*}={(T_k)}^{\upsilon}=\{t_k\}\times \{0,1\}$, 
\item $L_{\Upsilon_*}=\{*, 0,1\}^{\{0,1\}}$,
\item $\forall \lambda\in L_{\Upsilon_*}$, $\forall k\in\{0,1\}$, $({Id_{(T_k)}})^{\upsilon_*}_\lambda=Id_{({(T_k)}^{\upsilon_*})}$ (because ${\Upsilon_*}$ is functorial),
\item $\forall \lambda\in L_{\Upsilon_*}$, $\forall s\in\{0,1\}$,
\subitem  if $\lambda(s)=*$ then
$d^{\upsilon_*}_\lambda(t_0,s)=\emptyset $,
\subitem if $\lambda(s)\in\{0,1\}$ then, like with $\Upsilon$,
$d^{\upsilon_*}_\lambda(t_0,s)=\{(t_1,\lambda(s))\}$.
\end{itemize}

In other words, viewing $d^{\upsilon_*}_\lambda$ as a partial function, it is defined for $s\in\{0,1\}$ by :
\begin{itemize}
\item if $\lambda(s)=*$ then $(t_0,s)\notin \mathrm{Def}_{d^{\upsilon_*}_\lambda}$,
\item if $\lambda(s)\in\{0,1\}$ then $d^{\upsilon_*}_\lambda(t_0,s)=
(t_1,\lambda(s))$.
\end{itemize}

\paragraph{Realizations of $\Upsilon_*$.} As in the case of $\Upsilon$, we can write $\mathfrak{S}_{\Upsilon_*}$ as the set of all couples $(\lambda,\sigma)$
with $\lambda\in{L_{\Upsilon_*}}$ and
$\sigma=\underline{\emptyset}_{\Upsilon_*}$,
or
$\sigma\in\{0,1\}^{\{t_0\}}$, or
$\sigma\in\{0,1\}^{\{t_0,t_1\}}$ with $\sigma(t_1)=\lambda(\sigma(t_0))$ (which implies that $\lambda(\sigma(t_0))\neq *$). And we have
$
Z_{\Upsilon_*}=Z_\Upsilon.
$

\end{example}

\mbox{}

%\subsubsection{$\Gamma$, a hyper-deterministic intemporal lax-dynamic}

\begin{example}[$\Gamma$, a hyper-deterministic intemporal lax-dynamic]\label{exm Gamma}

This is a hyper-deterministic variant of the example \ref{exm Phi}, with the same set of states and the same clock, but depending on parameter values. Precisely, we set
\[
\Gamma=
\left((\gamma:\mathbf{1}\rightharpoondown \mathbf{Tran}^{\underrightarrow{\scriptstyle{L_\Gamma}}}) \stackrel{!}{\looparrowright}  (\xi_\mathbf{1}:\mathbf{1}\rightarrow \mathbf{Sets})\right)
\] where
\begin{itemize}
\item $st(\gamma)=\bullet^\gamma=\{0,1\}$, 
\item $L_\Gamma=\{a,b\}$, a set with two elements,
\item $\overrightarrow{0}^\gamma_a=Id_{\{0,1\}}$,
\item $\overrightarrow{0}^\gamma_b$ is defined as a transition%%%
\footnote{Equivalently, $\overrightarrow{0}^\gamma_b$ can be defined as a partial function by $0\notin \mathrm{Def}_{\overrightarrow{0}^\gamma_b}$ and $\overrightarrow{0}^\gamma_b(1)=1$.}
by
$\overrightarrow{0}^\gamma_b(0)=\emptyset$ and $\overrightarrow{0}^\gamma_b(1)=\{1\}$,
\item $\xi_\mathbf{1}$ is the canonical clock\footnote{See the example \ref{exm Phi}.} of $\mathbf{1}$ and $\gamma \stackrel{!}{\looparrowright} \xi_\mathbf{1}$ is the constant dynamorphism.
\end{itemize}

\paragraph{Realizations of $\Gamma$.} The set  $Z_\Gamma^*=\{0,1\}$ of nonempty outgoing realizations of $\Gamma$ is the same as for $\Phi$, but now we have
\[
\mathfrak{S}_\Gamma=
\{
(a,\underline{\emptyset}_\Gamma), 
(a,0),
(a,1),
(b,\underline{\emptyset}_\Gamma), 
(b,1)
\}.
\]

\end{example}

\mbox{}

%\subsubsection{$\mathbb{W} = $ \textcjheb{w}, an intemporal  open dynamic with functions as states}

\begin{example}[$\mathbb{W} = \textcjheb{w}$ , an intemporal  open dynamic with functions as states]\label{exm Vav}

The open dynamic $\mathbb{W}$ --- also denoted by the Hebrew letter \textcjheb{w} (vav) --- described in this example \ref{exm Vav} has been given in \cite{Dugowson:20160831} and \cite{DugowsonKlein:20190417} together with a dynamic denoted by $\mathbb{H}$  or by the Hebrew letter $\textcjheb{h}$ (hey) --- see \emph{infra}, example \ref{exm H0} --- and a third one denoted by $\mathbb{Y}$ or \textcjheb{y} (yod) (example \ref{exm Iod}) to produce the interactive family that we will describe in the example \ref{exm WHY family}, section \ref{subs examples of interactive families}. The choice of the Hebrew letter \textcjheb{w} comes from the fact that this dynamic is intended to (approximately and partially) model the philosophical concept that P. M. Klein \cite{KleinPM:2014} named in the same letter. The dynamic $\mathbb{W} = \textcjheb{w}$ is defined by
\[
\mathbb{W}=\left(
({\alpha_{\mathbb{W}}}:{\mathbf{1}}\rightharpoondown \mathbf{Tran}^{\underrightarrow{\scriptstyle{L_{\mathbb{W}}}}})
\stackrel{!}{\looparrowright}  
 (\xi_\mathbf{1}:\mathbf{1}\rightarrow \mathbf{Sets})
 \right),
\] 
where\footnote{For the meaning of $\mathcal{C}$, see notations in  the begining of the paper.}
\begin{itemize}
\item $st({\mathbb{W}})=\bullet^{\alpha_{\mathbb{W}}}=\mathcal{C}$,
\item $L_{\mathbb{W}}=\mathcal{C}$,
\item  for all $\lambda\in L_{\mathbb{W}}$, the transition $\overrightarrow{0}^{\alpha_{\mathbb{W}}}_\lambda$ is defined for all $f\in st({\mathbb{W}})$ by\\
\[
\overrightarrow{0}^{\alpha_{\mathbb{W}}}_\lambda(f)=
\left\lbrace
\begin{array}{l}
\{f\}\mathrm{\,\,if\,\,} f\diamondsuit \lambda,\\
\emptyset \mathrm{\,\,in\,\,other\,\,cases},
\end{array} 
\right.
\]
\item $\xi_\mathbf{1}$ is the canonical clock 
  of $\mathbf{1}$, and ${\alpha_{\mathbb{W}}} \stackrel{!}{\looparrowright} \xi_\mathbf{1}$ is the constant dynamorphism,
\end{itemize}
where $f\diamondsuit \lambda$ stands for $f_{\vert \mathrm{Def}_f\cap \mathrm{Def}_\lambda}=\lambda_{\vert \mathrm{Def}_f\cap \mathrm{Def}_\lambda}$.

\paragraph{Realizations of $\mathbb{W}$.}  For each $\lambda\in L_{\mathbb{W}}=\mathcal{C}$, 
the empty realization $\underline{\emptyset}_{\mathbb{W}_\lambda}=\underline{\emptyset}_{\mathbb{W}}$ is the partial function $st(\xi_\mathbf{1})=\{0\} \dashrightarrow st(\mathbb{W})=\mathcal{C}$ with an empty domain (or, as a transition, the map $0\mapsto \emptyset\subset\mathcal{C}$)
whereas 
a nonempty realization of $\mathbb{W}_\lambda$ can be identified with 
its value on the only instant $0\in st(\xi_\mathbf{1})$, this value being itself a real function $f\in\mathcal{C}$, possibly the empty real function $\underline{\emptyset}_\mathbf{R}$. 
Then, with this identification, we have
\[
\mathfrak{S}_\mathbb{W}
=
\bigcup_{\lambda\in\mathcal{C}}
\left(
\{
(\lambda,f), f\in\mathcal{C}, f\diamondsuit \lambda
\}
\cup
\{(\lambda,\underline{\emptyset}_{\mathbb{W}_\lambda})\}
\right).
\]
The set of nonempty outgoing realizations of $\mathbb{W}$ is then
\[
Z_\mathbb{W}^*=Z_\mathbb{W}\setminus\{\underline{\emptyset}_{\mathbb{W}}\}
=
\mathcal{C}.
\]
Note that the empty real function $\underline{\emptyset}_\mathbf{R}$ belongs to $Z_\mathbb{W}^*$.

\end{example}

\mbox{}
 
% \subsubsection{$\mathbb{H}=\textcjheb{h}$, a hyper-deterministic functorial open dynamic}

\begin{example}[$\mathbb{H}=  \textcjheb{h}$, a hyper-deterministic dynamic on $\mathbf{R}_+$]\label{exm H0}

The dynamic $\mathbb{H}$ --- also referred to as \textcjheb{h}, ``hey" in the Hebrew alphabet ---
has been introduced in \cite{Dugowson:20160831} and \cite{DugowsonKlein:20190417} under the name ``history" to constitute an interactive family together with $\mathbb{Y}=\textcjheb{y}$ (\emph{cf}. \emph{infra}, example \ref{exm Iod}) and  $\mathbb{W}=\textcjheb{w}$ (\emph{cf}. \emph{supra}, example \ref{exm Vav}). It is a hyper-deterministic functorial open dynamic with engine $(\mathbf{R}_+,+)$
and with a clock $\mathbf{h}_{]T_0,+\infty[}$ having instants $t\in ]T_0,+\infty[$ where $T_0$,
   called the \emph{origin of times}, is taken to be $\{-\infty\}\cup\mathbf{R}$. We distinguish the origin of times $T_0$ with the origin of histories which here will be taken to be $-\infty$.
More precisely, such a $T_0\in \{-\infty\}\cup\mathbf{R}$ having been chosen, we set
\[
{\mathbb{H}}={\textcjheb{h}}=
\left(
(
{(\mathbf{R}_+,+)}
\stackrel{{\alpha_\mathbb{H}}}\rightarrow 
\mathbf{Tran}^{\underrightarrow{\scriptstyle{{L_\mathbb{H}}}}}
)
\stackrel{\tau_\mathbb{H}}{\looparrowright}
{\mathbf{h}_{]T_0,+\infty[}}
\right),
\] 
with 
\begin{itemize}
\item $st(\alpha_\mathbb{H})=\bigcup_{t\in ]T_0, +\infty[}
\left(
\{t\}\times \mathcal{C}^1(]-\infty,t[)
\right)$,
\item  $st(\mathbf{h}_{]T_0,+\infty[})=]T_0,+\infty[$, 
\item ${L_\mathbb{H}}=\mathcal{C}^*_{]T_0,\rightarrow[}:=\bigcup_{u\in ]T_0, +\infty]}\mathcal{C}(]T_0,u[)$,
\item $\forall(t,f)\in st(\alpha_\mathbb{H})$, $\tau_\mathbb{H}(t,f)=t$,
\item  $\forall(t,f)\in st(\alpha_\mathbb{H})$, $\forall d\in\mathbf{R}^*_+, \forall u\in ]T_0,+\infty],\forall \lambda\in\mathcal{C}(]T_0,u[)$, 
\begin{itemize}
\item if $t+d\leq u$ and if there exists a (necessarily unique) $g\in \mathcal{C}^1(]-\infty, t+d[)$ such that $g_{\vert ]-\infty,t[}=f$ and $g_{\vert ]t,t+d[}=\lambda_{\vert ]t,t+d[}$, then we set $d^{\alpha_\mathbb{H}}_\lambda(t,f)=(t+d,g)$,
\item in all other cases, we set ${d^{\alpha_\mathbb{H}}_\lambda}((t,f))=\emptyset$ that is, viewing ${d^{\alpha_\mathbb{H}}_\lambda}$ as a partial function: $(t,f)\notin \mathrm{Def}_{d^{\alpha_\mathbb{H}}_\lambda}$.
\end{itemize}
\end{itemize}

\paragraph{Realizations of $\mathbb{H}$.}

It is easy to see that the outgoing part $\sigma$ of a nonempty realization $(\lambda,\sigma)\in \mathfrak{S}_\mathbb{H}$ can be uniquely represented by a real function of class $\mathcal{C}^1$ defined on an  interval of the form $]-\infty,a[$ or  $]-\infty,a]$, with $a>T_0$, that coincides with $\lambda$ on $]T_0,a[$. More precisely, with these representations, we verify that we can write
\[
\mathfrak{S}_\mathbb{H}=\{(\underline{\emptyset},\underline{\emptyset})\}\cup
\left[
\bigcup_{u\in]T_0,+\infty]}
\left(
\bigcup_{\lambda\in \mathcal{C}^1(]T_0,u[)}
(
\{\lambda\}
\times
\bigcup_{a\in]T_0,u]}
E_{\lambda, a}
)
\right)
\right]
\]
with $E_{\lambda, +\infty}=
\{\sigma\in \mathcal{C}^1(\mathbf{R}),
\sigma_{\vert ]T_0,+\infty[}=\lambda
\}$ whereas
\[E_{\lambda, a}=
\{
\sigma\in \mathcal{C}^1(]-\infty,a[)\cup \mathcal{C}^1(]-\infty,a]),
\sigma_{\vert ]T_0,a[}=\lambda_{\vert ]T_0,a[}
\}
\] when $a<+\infty$. 
Thus, the set of outgoing realizations of $\mathbb{H}$ is 
\[
Z_\mathbb{H}=\{\underline{\emptyset}\}
\cup
\left(
\bigcup_{a\in]T_0,+\infty]}\mathcal{C}^1(]-\infty,a[)
\right)
\cup
\left(
\bigcup_{a\in]T_0,+\infty[}\mathcal{C}^1(]-\infty,a])
\right).
\]
For any nonempty realization $\sigma\in{Z}_{\mathbb{H}}^*$ of $\mathbb{H}$, we  call the restriction ${\sigma}_{\vert ]-\infty,T_0]}$ the \emph{mythical part} of $\sigma$.
\end{example}

\mbox{}

\begin{example}[$\mathbb{Y}= $ \textcjheb{y}, a non-deterministic functorial mono-dynamic]\label{exm Iod}

Introduced in \cite{Dugowson:20160831} and \cite{DugowsonKlein:20190417} as a ``future" dynamic together with $\mathbb{W}$ (\emph{cf}. \emph{supra}, example \ref{exm Vav}) and $\mathbb{H}$ (example \ref{exm H0}), the dynamic that we designate by $\mathbb{Y}$ or \textcjheb{y} (yod) and  call a ``lipschitzian source", is defined by
\[
{\mathbb{Y}}=
\left(
({\alpha_{\mathbb{Y}}}:{(\mathbf{R}_+,+)}\rightarrow \mathbf{Tran})
\stackrel{\tau_\mathbb{Y}}{\looparrowright} 
\xi_{\mathbf{R}_+}
\right),
\] 
where
\begin{itemize}
\item  ${\xi_{\mathbf{R}_+}}$ is the existential clock associated with the monoïd $(\mathbf{R}_+,+)$, that is such that $st({\xi_{\mathbf{R}_+}})=\mathbf{R}_+$ and  $d^{\xi_{\mathbf{R}_+}}(t)=t+d$ for all instants $t\in \mathbf{R}_+$ and all durations $d\in\mathbf{R}_+$,
\item the set of states is $st(\alpha_{\mathbb{Y}})=\mathbf{R}_+\times \mathbf{R}$,
\item for all states $(t,a)\in st(\alpha_{\mathbb{Y}})$, 
\subitem $\tau_\mathbb{Y}(t,a)=t$,
\subitem and for all $d\in\mathbf{R}_+$, $d^{\alpha_{\mathbb{Y}}}(t,a)=\{t+d\}\times [a-d,a+d]$.
\end{itemize}

\paragraph{Realizations of $\mathbb{Y}$.}
It is immediate to see that a realization $\sigma\in Z_\mathbb{Y}=\mathfrak{S}_\mathbb{Y}$ is a partial function $\mathbf{R}_+\dashrightarrow \mathbf{R}_+\times \mathbf{R}$ defined on an interval $D$ of the form $[0,a]$ or $[0,a[$ that can be identified with a metric map\footnote{That is a Lipschitz function with Lipschitz constant $1$.} $\sigma:D\rightarrow\mathbf{R}$:
\[
Z_\mathbb{Y}\simeq
\bigcup_{a\in\mathbf{R}_+\cup \{+\infty\}}Lip_1([0,a[)
\cup
\bigcup_{a\in\mathbf{R}_+} Lip_1([0,a])
\]
where $Lip_1(D)=\{\sigma:D\rightarrow\mathbf{R}, \sigma\mathrm{\,is\,a\,metric\,map}\}$.
\end{example}

\subsection{Some relations with Bastiani (-Ehresmann)'s control systems}\label{subs relation with Ehresmann control systems}

In her article \cite{Bastiani_Ehresmann:196505}, published in 1967, Andrée Bastiani (-Ehresmann) considered some \emph{control systems} --- called \emph{systèmes guidables} in French --- which, leaving aside topological aspects, seem to be quite close to some of our open systems 
which we have developed, as indicated in the introduction to this paper, with a view to proposing a theory of interactivity, from 
a first categorical generalization of some ``closed" dynamical systems --- namely  mono-dynamics on monoids (see section \ref{subsubs mono-dynamics}) --- a generalization itself prompted by an oral remark by Mme Ehresmann.
To lay the foundations for a further exploration of the possible connections between these two notions, we reformulated in our own language Bastiani (-Ehresmann)'s definitions, which was originally given in the language and notations introduced by Charles Ehresmann in his book  \emph{Catégories et Structures} \cite{EhresmannCharles:1965} and which have been more recently rapidly mentioned again by Mme Ehresmann in two lectures \cite{Ehresmann:201006,Ehresmann:201106}, with more current notations. Leaving aside, as announced, topological aspects, it then turns out that the definition of a control system given by Mme Ehresmann is equivalent to considering the data $(F,q)$ of a disjunctive functor $\mathbf{G}\stackrel{F}{\rightarrow}\mathbf{ParF}$ and of a functor $\mathbf{G}\stackrel{q}{\rightarrow} \mathbf{H}$, with $\mathbf{G}$ and $\mathbf{H}$ some categories which we will assume to be small. Intuitively,  the objects of $\mathbf{H}$ can be seen as  \emph{instants}, and its arrows as \emph{durations} whereas the objects of $\mathbf{G}$ can be seen as ``parameterized instants" and its arrows as ``parameterized durations".
With the functor $F$ is then defined a set $E:= \bigsqcup_{g\in\dot{\mathbf{G}}} F(g)$ whose elements we shall see as ``parameterized states",  and a partial action of $\mathbf{G}$  on $E$ given  for all $\gamma\in\overrightarrow{\mathbf{G}}$ and all $e\in \mathrm{Def}_{F(\gamma)}\subseteq F(\mathrm{dom}(\gamma))\subseteq E$ by $\gamma.e:=F(\gamma)(e)$.
A \emph{solution on a subcategory $\mathbf{S}\subseteq \mathbf{H}$} of the control system $(F,q)$
then consists in a  couple $(\dot{\mathbf{S}}\stackrel{\varphi}{\rightarrow}E,\mathbf{S}\stackrel{\psi}{\rightarrow}\mathbf{G})$
where   $\psi:\mathbf{S}{\rightarrow}\mathbf{G}$ is a functor and $\varphi$ is a map that  associates with each instant $t\in \dot{\mathbf{S}}$ a parameterized state $\varphi(t)\in E$, 
 such that we have $ \dot{q}\circ p\circ \varphi=Id_{\dot{\mathbf{S}}}$, $ q\circ \psi=Id_{\overrightarrow{\mathbf{S}}}$ and,  for all $h\in\overrightarrow{\mathbf{S}}$,
$\psi(h).\varphi(t_1)=\varphi(t_2)$ where $t_1=\mathrm{dom}(h)$ and $t_2=\mathrm{cod}(h)$.

An interpretation of these definitions in relation with ours is given by the following association with each functorial hyper-deterministic open dynamic 
\[
A=\left((\alpha:\mathbf{D}\rightarrow \mathbf{ParF}^{\underrightarrow{\scriptstyle{L}}}) \stackrel{\rho}{\looparrowright}  (\mathbf{h}:\mathbf{D}\rightarrow \mathbf{Sets})\right)
\]
of a Bastiani (-Ehresmann)'s control system 
\[GS(A)=(\mathbf{G}\stackrel{F}{\rightarrow}\mathbf{ParF},\mathbf{G}\stackrel{q}{\rightarrow} \mathbf{H}),\]
namely the one given by 
\begin{itemize}
\item $\dot{\mathbf{H}}:=st(\mathbf{h})$,
\item $\overrightarrow{\mathbf{H}}:=\{(t_1,d,t_2)\in \dot{\mathbf{H}}\times\overrightarrow{\mathbf{D}}\times \dot{\mathbf{H}}, d^\mathbf{h}(t_1)=t_2\}$, with obvious source, target and composition,
\item $\mathbf{G}:=L\times\mathbf{H}$ (where $L$ is seen as a discrete category), 
\item $q:\mathbf{G}{\rightarrow} \mathbf{H}$ is the projection on $\mathbf{H}$, that is the forgetting of the parameter:
\[
q\left((\lambda,t_1)\stackrel{(\lambda,t_1,d,t_2)}{\longrightarrow}(\lambda, t_2)\right)
:=
\left(t_1\stackrel{(t_1,d,t_2)}{\longrightarrow}t_2\right),
\]
\item for all $(\lambda,t)\in\dot{\mathbf{G}}$,  $F(\lambda,t):=\{\lambda\}\times \rho^{-1}(t)\subset L\times st(\alpha)$,
\item for all $((\lambda,t_1)\stackrel{(\lambda,t_1,d,t_2)}{\longrightarrow}(\lambda, t_2))\in\overrightarrow{\mathbf{G}}$ and all $s\in\rho^{-1}(t_1)$,  
\subitem - if $s\in \mathrm{Def}_{d^\alpha_\lambda}$ then
$F(\lambda,t_1,d,t_2)(\lambda,s):=(\lambda, d^\alpha_\lambda(s))$,
\subitem - else $(\lambda,s)\notin \mathrm{Def}_{F(\lambda,t_1,d,t_2)}$.
\end{itemize}

It is then straightforward to verify that every realization $(\lambda,\sigma)\in\mathfrak{S}_A$ 
gives a solution $(\varphi,\psi)$ 
of the  control system $GS(A)$ over the full subcategory 
$\mathbf{S}\subseteq \mathbf{H}$ 
defined by $\dot{\mathbf{S}}=\mathrm{Def}_\sigma$, namely the couple $(\varphi,\psi)$ given by $\varphi(t)=(\lambda,\sigma(t))\in E=\bigsqcup_{g\in\dot{\mathbf{G}}}F(g)$ for every $t\in\dot{\mathbf{S}}$ and  $\psi(t_1,d,t_2)=(\lambda,t_1,d,t_2)\in \overrightarrow{\mathbf{G}}$
for every $(t_1,d,t_2)\in\overrightarrow{\mathbf{S}}$.

The association $A\mapsto GS(A)$  gives us a first idea of the possible relationships between our open systems and Mme Ehresmann's control systems, each with their own limitations. Let us make a few comments on this. First, not that  $GS$ is not injective (up to isomorphism)
since, for example
\begin{itemize}
\item the open dynamic
$A=\left(
(\xi:(\mathbf{R}_+,+)\rightarrow \mathbf{Sets}) 
\stackrel{Id}{\looparrowright}  
(\xi:(\mathbf{R}_+,+)\rightarrow \mathbf{Sets}) 
\right)$, where $st(\xi)=\mathbf{R}_+$ and, for every $d\in\mathbf{R}_+$ and every $t\in\mathbf{R}_+$, $d^\xi(t)=t+d$, 
\item and the open dynamic
$B=\left(
(\zeta:(\mathbf{R}_+,\leq)\rightarrow \mathbf{Sets}) 
\stackrel{Id}{\looparrowright}  
(\zeta:(\mathbf{R}_+,\leq)\rightarrow \mathbf{Sets}) 
\right)$, where for every $t\in\mathbf{R}_+$ we have $t^\xi=\{t\}$  and, for every $d=(t_1\leq t_2)\in \overrightarrow{(\mathbf{R}_+,\leq)}$, we have $d^\zeta(t_1)=t_2$, 
\end{itemize}
are not isomorphic, but $GS(A)$ and $GS(B)$ are essentially the same control systems, the important difference between the categories $(\mathbf{R}_+,+)$ and $(\mathbf{R}_+,\leq)$ being lost in translation. 

In addition, while our open dynamics are not necessarily deterministic whereas Mme Ehresmann's control systems could be said to be hyper-deterministic,
the formulation we obtained of a control system as a couple $(\mathbf{G}\stackrel{F}{\rightarrow}\mathbf{ParF},{\mathbf{G}\stackrel{q}{\rightarrow} \mathbf{H})}$ suggests a non-deterministic generalization, given by couples of the form $(\mathbf{G}\stackrel{F}{\rightarrow}\mathbf{Trans},\mathbf{G}\stackrel{q}{\rightarrow} \mathbf{H})$. As we said in our introduction,  the necessity to use lax-functors instead of functors in our own theory came from our treatment of interactivity (cf. theorem \ref{thm stability}). If a theory of interacting ``control systems'' would be developed, it could lead as well to consider lax-functorial non-deterministic systems given by couples of the form 
$(\mathbf{G}\stackrel{F}{\rightharpoondown}\mathbf{Trans},\mathbf{G}\stackrel{q}{\rightarrow} \mathbf{H})$ which could be the subject of further research. 

On the other side, $GS$ isn't surjective either. In particular, note that 
the design of control systems gives to their ``parametrical" aspects ---  which are implied both in the category $\mathbf{G}$ of ``parameterized instants" and in the set $E$ of ``parameterized states"  --- a local nature, as opposed to the parameters of our open dynamics, which are on the contrary global in nature, and this can be viewed as an advantage of control systems. 

Finally, according to our definition of the realizations of an open dynamic, note that even in the case when a control system $G$ is of the form $GS(A)$ with $A$ an open dynamic  in the sense of our theory, a solution of $G$ that is defined over a subcategory $\mathbf{S}\subseteq\mathbf{H}$ that does not satisfy the property
\[\forall(t_1,d,t_2)\in\overrightarrow{\mathbf{H}},\,
t_2\in\dot{\mathbf{S}}
\Rightarrow 
 (t_1,d,t_2)\in\overrightarrow{\mathbf{S}}\]
cannot be obtained from a realization of $A$ : thanks to a greater partiality, Bastiani (-Ehresmann)'s control systems have more solutions\footnote{At least as they are defined in \cite{Bastiani_Ehresmann:196505} since this notion of partial solutions does not appear at all in \cite{Ehresmann:201006} and \cite{Ehresmann:201106}.} than our open dynamics, and this can be seen as another advantage of control systems. Of course, it would be easy to broaden in turn our definition of realizations of open dynamics to include more partiality, but the real difficulties will then arise in interacting with other dynamics: how can a complex system work  when some of its components are removed or added ? This type of question, linked to the philosophical problem known as the ``Ship of Theseus", seems to us to be at the core of Andrée Ehresmann's research work, but in its current state our own theory does not yet allow us to address it correctly since our collective \emph{global dynamics} need all their components to be ``simultaneously"\footnote{For some \emph{synchronization}.} active to obtain a realization defined at the corresponding instant.

\section{Interactive families}\label{sec interactive families}

The main purpose of this section is to give the definition of interactive families, namely interacting families of open dynamics. For this, we firstly  give some reminders about \emph{binary relations}, \emph{multiple relations} and \emph{multiple binary relations} (\textbf{§\,\ref{subs reminders multiple binary relations}}), then  we give the definitions of an \emph{interaction request} and of an \emph{interaction relation} between some open dynamics (\textbf{§\,\ref{subs interaction relations in a family}}) and the definition of a \emph{synchronization} between these dynamics  (\textbf{§\,\ref{subs synchronization}}). An interaction request (or an interaction relation) and a synchronization then define an \emph{interaction} in the family of open dynamics under consideration, and such an interaction --- together with a third element, called \emph{privacy} or \emph{social mode} --- leads in turn  to the definition of an \emph{interactive family} (\textbf{§\,\ref{subs familles interactives}}). In \textbf{§\,\ref{subs connectivity structures of an IF}}, we associate four connectivity structures with any given interactive family, in particular the \emph{realization connectivity structure of the interaction relation}, which is the most important and which we  simply call \emph{the connectivity structure} of the considered interactive family. Finally, in \textbf{§\,\ref{subs examples of interactive families}}, we give some examples of interactive families.

\subsection{Binary, multiple and multiple binary relations}\label{subs reminders multiple binary relations}

\subsubsection{Binary relations}

Given $E$ and $E'$ two sets, a binary relation $B$ from $E$ to $E'$ is defined by its \emph{domain} $E=\mathrm{dom}(B)$, its \emph{codomain} $E'=\mathrm{cod}(B)$ and its \emph{graph} $\vert B\vert\subset E\times E'$. According to the introduction of the paper, we  also consider such a binary relation as a (not necessarily deterministic) transition $E\rightsquigarrow E'$, that is a map $E\rightarrow \mathcal{P}(E')$, setting for any $e\in E$
\[
B(e)=\{e'\in E', (e,e')\in \vert B\vert\}.
\]
Then
the \emph{image} of $B$ is defined by
\[\mathrm{Im}(B)
=\bigcup_{e\in E} B(e)\subset E',\] 
the \emph{converse binary relation}, denoted as $B^{-1}$ or  $B^\top$, is
defined by its graph 
\[
\vert B^\top\vert:=\vert B\vert^\top=
\{(e,e')^\top, (e,e')\in \vert B\vert\}
\quad\mathrm{where}\,(e,e')^\top:=(e',e)
\]
 or, equivalently, by
\[
\forall e'\in E',
B^\top(e')
=\{e\in E, B(e)\ni e'\},
\] 
and the \emph{domain of definition} of $B$ is given by
\[
\mathrm{Def}_{B}=\{e\in E, B(e)\neq\emptyset\}=\mathrm{Im}(B^\top).
\]

The set of binary relations from $E$ to $E'$ is denoted by $\mathbf{BR}_{(E,E')}$, and the class of all binary relations is denoted by $\mathbf{BR}$.

\subsubsection{Multiple relations}

In this section and the next, we recall  the definitions we gave in \cite{Dugowson:20150505}  and \cite{Dugowson:20150807} about \emph{multiple relations} and \emph{multiple binary relations}\footnote{Or, as well: \emph{binary multiple relations}.}. Given 
 $\mathcal{E}=(E_i)_{i\in I}$ a family of sets indexed by a set $I$,
 the product  $\prod_{i\in I} E_i$  is also denoted as $\Pi_I\mathcal{E}$ or $\Pi\mathcal{E}$.

\begin{df}\label{df relation multiple} A multiple relation  $R$ is the data  $R=(I,\mathcal{E},\vert R\vert)$ of
\begin{itemize}
\item a set $I=\mathrm{ar}(R)$, called the \emph{index set}  or the \emph{arity} of $R$,
\item an $I$-family of sets $\mathcal{E}=(E_i)_{i\in I}$ called the \emph{context of $R$}, 
\item a subset $\vert R\vert\subseteq\Pi_I\mathcal{E}$, called the \emph{graph of $R$}.
\end{itemize}
\end{df}

 The class of all multiple relations with a given index set $I$ --- which are also  called \emph{$I$-relations} --- is denoted by $\mathbf{MR}_I$. Given a context $\mathcal{E}=(E_i)_{i\in I}$ on $I$, 
  the set of multiple relations with context $\mathcal{E}$ is denoted as $\mathbf{MR}_{\mathcal{E}}$. 
For example, if $2$ denotes the set $\{0,1\}$, 
the class $\mathbf{MR}_2$ can be seen as the class $\mathbf{BR}$ of all binary relations between sets and, given  $(E_0,E_1)$ a couple of sets, we have $\mathbf{MR}_{(E_0,E_1)}=\mathbf{BR}_{(E_0,E_1)}$.

If $R$ and $S$ are multiple relations in a context $\mathcal{E}$, we'll denote by $R\cap S$ their intersection, that is the multiple relation in the same context such that $\vert R\cap S\vert=\vert R\vert \cap \vert S\vert$, and we define an order $(\mathbf{MR}_{\mathcal{E}},\subseteq)$ by putting $R\subseteq S$ when $R=R\cap S$. 
If $J\subseteq I$, we put
 $\mathcal{E}_{\vert J}=(E_j)_{j\in J}$,  $\Pi_J\mathcal{E}= \Pi(\mathcal{E}_{\vert J})=\prod_{j\in J} E_j$ and we designate by $0_J$ the minimum element of $(\mathbf{MR}_{\mathcal{E}_{\vert J}},\subseteq)$, that is the empty $J$-relation $0_J=(J,\mathcal{E}_{\vert J},\emptyset)$, and by $1_J$ its maximum element, that is the plain $J$-relation  $1_J=(J,\mathcal{E}_{\vert J},\Pi_J\mathcal{E})$.  
Note that in the case where $J=\emptyset$, we have $0_\emptyset\neq 1_\emptyset$, since the graph of $1_\emptyset$ is a singleton $\Pi_\emptyset\mathcal{E}=\{\bullet\}$, whereas $\vert 0_J\vert=\emptyset$.
If $R\in\mathbf{MR}_{\mathcal{E}_{\vert J}}$, we also denote by $R_{\vert K}$ 
the restriction of $R$ on $K\subseteq J$, that is the $K$-relation defined by 
$R_{\vert K}:=(K,\mathcal{E}_{\vert K}, \vert R\vert_{\vert K})$, where $\vert R\vert_{\vert K}=\{y_{\vert K}, y\in\vert R\vert\}$ or, equivalently,
$\vert R\vert_{\vert K}=\{x\in\Pi_K\mathcal{E}, \exists y\in\vert R\vert, \forall k\in K, x_k=y_k\}$. 
Finally, we denote  by
$\mathbf{MR}_{\subseteq\mathcal{E}}$ the set of multiple relations  \emph{inside} the context $\mathcal{E}$, that is the set of all multiple  relations  $R=(J,\mathcal{E}_{\vert J},\vert R\vert)$ with
 $J\subseteq I$  and
$\vert R\vert\subseteq\Pi_J\mathcal{E}$. In other words 
$\mathbf{MR}_{\subseteq\mathcal{E}} 
=\bigcup_{J\subseteq I}\mathbf{MR}_{\mathcal{E}_{\vert J}}
\subset \bigcup_{J\subseteq I}\mathbf{MR}_J$.
The set  $\mathbf{MR}_{\subseteq\mathcal{E}}$ can be endowed with a ``gluing operator" $\otimes$ defined\footnote{See \cite{Dugowson:20150505}, section § \textbf{1.5.1}, where it was denoted by $\bowtie$.} for a $J_1$-relation $R_1$ and a $J_2$-relation $R_2$ as the $(J_1\cup J_2)$-relation $R_1\otimes R_2$ containing all ``glued" families $x_1+x_2$ with some \emph{compatible} $x_n\in \vert R_n\vert$, that is such that $x_1$ and $x_2$ have the same restrictions on $J_1\cap J_2$. In other words, for every $x\in\Pi_{J_1\cup J_2}\mathcal{E}$, we have $x\in\vert R_1\otimes R_2\vert$ if and only if $x_{\vert J_1}\in\vert R_1\vert$ and $x_{\vert J_2}\in\vert R_2\vert$.
Note also that the relation $\mathbf{1}=1_\emptyset$ is  neutral for this operator, giving $(\mathbf{MR}_{\subseteq\mathcal{E}}, \otimes, \mathbf{1})$ a structure of a commutative monoïd,  whereas $0_I$ is an  annihilating element. 

\begin{rmq} The intersection of two multiple relations in a given context is nothing but a peculiar case of the gluing operator $\otimes$ applied to relations inside a same context and having a same arity.
\end{rmq}

%
%
%\begin{lm}\label{lemma otimes and cap for multiple relations} Given a context $\mathcal{E}=(E_i)_{i\in I}$, two subsets $J_1\subseteq I$ and $J_2\subseteq I$, and four multiple relations $R_1$, $S_1$, $R_2$ and $S_2$ belonging in $\mathbf{MR}_{\subseteq\mathcal{E}}$ such that $\mathrm{ar}(R_1)=\mathrm{ar}(S_1)=J_1$ and $\mathrm{ar}(R_2)=\mathrm{ar}(S_2)=J_2$,  we have
%\[
%(R_1\cap S_1)\otimes (R_2\cap S_2)=(R_1 \otimes R_2)\cap(S_1 \otimes S_2).
%\]
%\end{lm}
%\noindent\textbf{Proof.} The operation  $\cap$ is a special case of $\otimes$ when the relations at stake have the same arity, therefore the above formula is simply equivalent to 
%\[(R_1\otimes S_1)\otimes (R_2\otimes S_2)=(R_1 \otimes R_2)\otimes(S_1 \otimes S_2).\]
%\begin{flushright}$\square$\end{flushright}
%

\subsubsection{Multiple binary relations}\label{subsubs multiple binary relation}

\begin{df}\label{df multiple binary relation} A \emph{multiple binary relation} $Q$ is the data $(I,\mathcal{W},\mathcal{M},\vert Q \vert)$ of
\begin{itemize}
\item a set $I=\mathrm{ar}(Q)$, called the \emph{index set}  or the \emph{arity} of $Q$,
\item an $I$-family of sets $\mathcal{W}=(W_i)_{i\in I}$ called the \emph{incoming context of $Q$}, 
\item an $I$-family of sets $\mathcal{M}=(M_i)_{i\in I}$ called the \emph{outgoing context of $Q$}, 
\item a subset $\vert Q \vert\subseteq\Pi_I\mathcal{E}$ called the \emph{graph of $Q$}, where $\mathcal{E}=(E_i)_{i\in I}$ is given by $E_i=W_i\times M_i$ for all $i\in I$ and is called the \emph{product  context of $Q$}.
\end{itemize}
\end{df}

 The class of all multiple binary relations with a given index set $I$ --- which are also called \emph{$I$-multiple binary relations} or \emph{$I$-binary relations} --- is denoted by $\mathbf{MBR}_I$. The set of all multiple binary relations with  given incoming context  $\mathcal{W}=(W_i)_{i\in I}$ and outgoing context $\mathcal{M}=(M_i)_{i\in I}$ is denoted by $\mathbf{MBR}_{(\mathcal{W}, \mathcal{M})}$ and, as in the case of multiple relations, we'll denote $\mathbf{MBR}_{\subseteq(\mathcal{W}, \mathcal{M})}$ the set of multiple binary relations  \emph{inside} the context $(\mathcal{W}, \mathcal{M})$, that is the set of all multiple  relations  $R=(J,\mathcal{W}_{\vert J}, \mathcal{M}_{\vert J}, \vert R\vert)$ with
 $J\subseteq I$, 
and
$\vert R\vert
\subseteq \Pi_J\mathcal{E}
=\prod_{j\in J} {(W_j\times M_j)}$. A gluing operator $\otimes$ is defined on $\mathbf{MBR}_{\subseteq(\mathcal{W}, \mathcal{M})}$ exactly in the same way that for $\mathbf{MR}_{\subseteq\mathcal{E}}$: if, for $n\in\{1,2\}$, we have $R_n=(J_n, \mathcal{W}_{\vert J_n}, \mathcal{M}_{\vert J_n}, \vert R_n\vert)$, then $R_1\otimes R_2$  designates the multiple binary relation $R$  with arity $J=J_1\cup J_2$, with context $(\mathcal{W}_{\vert J}, \mathcal{M}_{\vert J})$ and with graph 
%the set $\vert R\vert$ containing  all ``glued" families $y_1+y_2$ with some compatible $y_n\in \vert R_n\vert$, that is with $y_1$ and $y_2$ having the same restrictions on $J_1\cap J_2$.
$\vert R\vert
=\{y\in  \Pi_J\mathcal{E}, 
\forall n\in\{1,2\}, 
y_{\vert J_n}\in\vert R_n\vert\}$.
When $R_1$ and $R_2$ have the same arity, 
$R_1\otimes R_2$ can be simply denoted by 
$R_1\cap R_2$, 
and we obtain an order on 
$\mathbf{MBR}_{(\mathcal{W}_{\vert J}, \mathcal{M}_{\vert J})}$ by putting 
$R_1\subseteq R_2$ iff $R_1\cap R_2=R_1$.

\subsubsection{Type conversions between $\mathbf{MBR}_I$, $\mathbf{MR}_{2I}$, $\mathbf{MR}_I$ and $\mathbf{BR}$}\label{conversions between MR, MBR and BR}

With any $Q=(I,\mathcal{W},\mathcal{M},\vert Q\vert)\in \mathbf{MBR}_I$, we associate the $I$-multiple relation $mr(Q):= (I,\mathcal{E}, \vert Q\vert)\vert$ where $\mathcal{E}=(W_i\times M_i)_{i\in I}$. Note that the gluing operator $\otimes$ defined on $\mathbf{MBR}_{\subseteq(\mathcal{W}, \mathcal{M})}$ can then be defined from the operator $\otimes$ defined on $\mathbf{MR}_{\subseteq \mathcal{E}}$ by the fact that, for $R_1$ and $R_2$ belonging to $\mathbf{MBR}_{\subseteq(\mathcal{W}, \mathcal{M})}$, we have $mr(R_1\otimes R_2)=mr(R_1)\otimes mr(R_2)$. Of course, if $R_1$ and $R_2$ have the same arity, 
we also have  $mr(R_1\cap R_2)=mr(R_1)\cap mr(R_2)$.
With each $Q\in \mathbf{MBR}_I$, we also associate the binary relation\footnote{Recall that we often see binary relations as (not necessarily deterministic) transitions.} $br(Q):\Pi_I \mathcal{W}\rightsquigarrow \Pi_I\mathcal{M}$ that has graph $\vert br(Q)\vert$ given by $\vert Q\vert$ after an obvious re-indexing. Note that the applications $mr:\mathbf{MBR}_I\rightarrow\mathbf{MR}_I$ and $br:\mathbf{MBR}_I\rightarrow\mathbf{BR}$ so defined are injective on non-empty relations\footnote{Because if $\vert Q\vert\neq\emptyset$, then $\Pi_I\mathcal{W}\times\Pi_I\mathcal{M}\neq\emptyset$ and, in this case, this product characterizes all sets $W_i$ and $M_i$.}. 
In particular, we will often define the graph $\vert Q\vert$ of a multiple binary relation $Q\in \mathbf{MBR}_{(\mathcal{W}, \mathcal{M})}$ by giving, for all $w\in \Pi_I\mathcal{W}$, the set
$br(Q)(w)\subset\Pi_I\mathcal{M}$.

Note also that, applying notations for binary relations, we have:
\[\mathrm{Im}(br(Q))
=\bigcup_{w\in\Pi_I(\mathcal{W})} br(Q)(w)\subset\Pi_I\mathcal{M},\] 
\[\forall \mu\in\Pi_I\mathcal{M}, br(Q)^\top(\mu)=\{w\in\Pi_I(\mathcal{W}), br(Q)(w)\ni\mu\},\] 
and  
\[
\mathrm{Def}_{br(Q)}=\{w\in\Pi_I(\mathcal{W}), br(Q)(w)\neq\emptyset\}=\mathrm{Im}(br(Q)^\top).
\]

Moreover, by putting $2I=I_0\cup I_1$ where, for $k\in\{0,1\}$, $I_k=I\times\{k\}$, we define canonical reciprocal bijections 
 \[mr_2:\mathbf{MBR}_I\leftrightarrow\mathbf{MR}_{2I}:mbr\]
  in a trivial way: for any $Q=(I,\mathcal{W},\mathcal{M},\vert Q\vert)\in \mathbf{MBR}_I$, where $\mathcal{W}=(W_i)_{i\in I}$ and $\mathcal{M}=(M_i)_{i\in I}$, we set $mr_2(Q)=(2I,\mathcal{D},\widetilde{\vert Q\vert})$ where
$\mathcal{D}=(D_j)_{j\in I_0\cup I_1}$ with, for each $i\in I$, $D_{(i,0)}=W_i$ and $D_{(i,1)}=M_i$,
and $\widetilde{\vert Q\vert}$ is the image of $\vert Q\vert$ given by the canonical bijection $\Pi_I(W_i\times M_i)\rightarrow \Pi_{2I}\mathcal{D}$.

\subsection{Interaction relations in a family of open dynamics}
\label{subs interaction relations in a family}

From now on, $I$ denotes a non-empty set and  $\mathcal{A}=(A_i)_{i\in I}$ an $I$-family 
of 
open dynamics  $A_i=\left((\alpha_i:\mathbf{D}_i\rightharpoondown \mathbf{Tran}^{\underrightarrow{\scriptstyle{L_i}}}) \stackrel{\rho_i}{\looparrowright}  (\mathbf{h}_i:\mathbf{D}_i\rightarrow \mathbf{Sets})\right)$.
For each $i\in I$, the set $Z_{A_i}$ of outgoing realizations of $A_i$ is simply  denoted by $Z_i$ --- thus $Z_i^*$  denotes the set of nonempty realizations of $A_i$ --- and, for any $\lambda\in L_i$, the set of (outgoing) realizations of the open mono-dynamic $(A_i)_\lambda$ is denoted by $Z_{i,\lambda}$ instead of $ Z_{(A_i)_\lambda}$.
We also put 
$\mathcal{Z}:=({Z}_{i})_{i\in I}$, 
$\mathcal{Z}^*:=({Z}^*_{i})_{i\in I}$,
$\mathcal{L}:=(L_i)_{i\in I}$ and 
$\mathcal{E}:=(E_i)_{i\in I}$ 
where, for each $i\in I$, $E_i:={Z}_i\times L_i$.
The elements $\mathfrak{q}$ of $\Pi_I\mathcal{E}\backsimeq \Pi_I\mathcal{Z}\times \Pi_I\mathcal{L}$ are often  denoted as in the following form:
\begin{equation}\label{eq form of families in a request}
\mathfrak{q}=
\left(
\begin{array}{c}
\lambda_i  \\ 
\sigma_i 
\end{array} 
\right)_{i\in I}
\end{equation}
 with, for all $i\in I$, $\sigma_i\in Z_i$ and $\lambda_i\in L_i$. The coefficients of  $\mathfrak{q}$ is sometimes  designated for all $i\in I$ by $\mathfrak{q}_i:=\sigma_i$ and $\mathfrak{q}^i:=\lambda_i$.
With these notations, such a $\mathfrak{q}\in\Pi_I\mathcal{E}$ is said to be
\emph{coherent (for the family $\mathcal{A}$)} if, for all $i\in I$,
$\mathfrak{q}_i\in {Z}_{i,\mathfrak{q}^i}$. More generally, a set $C\subset\Pi_I\mathcal{E}$ 
is said to be \emph{coherent} if all its elements are coherent, and a multiple binary relation $Q\in \mathbf{MBR}_{(\mathcal{Z}, \mathcal{L})}$ is said to be \emph{coherent} if its graph $\vert Q\vert$ is also coherent. The multiple binary relation whose graph is the maximal coherent one is denoted $\Omega_\mathcal{A}$. Then a multiple binary relation $Q\in \mathbf{MBR}_{(\mathcal{Z}, \mathcal{L})}$ is coherent if $Q\subseteq \Omega_\mathcal{A}$, that is if  $\vert Q\vert \subseteq\vert \Omega_{\mathcal{A}}\vert$.
Remark that
\[
\vert \Omega_{\mathcal{A}}\vert=
(\Pi_I \mathfrak{S}_i)^\top:=
\{\mathfrak{q}\in \Pi_I\mathcal{E}, \mathfrak{q}^\top\in \Pi_I\mathfrak{S}_i\},
\]
where $\mathfrak{q}^\top:=\left(
\begin{array}{c}
\sigma_i  \\ 
\lambda_i 
\end{array} 
\right)_{i\in I}\in\prod_{i\in I}(L_i\times Z_i)$.
With every multiple binary relation $Q\in \mathbf{MBR}_{(\mathcal{Z}, \mathcal{L})}$ we associate its \emph{coherent part}
$\widecheck{Q}:= Q\cap \Omega_\mathcal{A}$, that is the multiple binary relation such that
$\vert \widecheck{Q} \vert = \vert Q\vert\cap \vert \Omega_{\mathcal{A}}\vert$.

In the following, multiple binary relations $Q\in \mathbf{MBR}_{(\mathcal{Z}, \mathcal{L})}$ are also  called \emph{interaction requests} for the family $\mathcal{A}=(A_i)_{i\in I}$. Such a request $Q$  is said to be 
\begin{itemize}
\item \emph{normal} if $\mathrm{Def}_{br(Q)}\supseteq\Pi_I\mathcal{Z}^*$, 
\item \emph{admissible} if $\widecheck{Q}\neq\emptyset$,
\item \emph{functional} if $br(Q)$ is a (partial) function $\Pi_I\mathcal{Z}\dashrightarrow \Pi_I\mathcal{L}$, that is if for every $(\sigma_i)_{i\in I}\in \Pi_I\mathcal{Z}$ one has
\[card(br(Q)((\sigma_i)_{i\in I}))\leq 1,\] 
\item \emph{strongly functional} if for every $i\in I$ and for every $(\sigma_j)_{j\in I\setminus \{i\}}\in \Pi_{j\neq i} Z_j$, one has 
\[card 
(\{
\lambda_i\in L_i,
\exists \mathfrak{q}\in\vert Q\vert,
(\forall j\neq i, \mathfrak{q}_j=\sigma_j)
\,\mathrm{and}\,
\mathfrak{q}^i=\lambda_i
\})
\leq 1.\]
\end{itemize}

We'll denote by $\mathbf{MBR}^\#_{(\mathcal{Z}, \mathcal{L})}$ the set of admissible interaction requests.

\begin{df} \label{df interaction relation}[Interaction relations]
An \emph{interaction relation for $\mathcal{A}$} is a coherent interaction request for $\mathcal{A}$.    
\end{df}

The set of interaction relations for the family $\mathcal{A}$ of open dynamics is denoted by $\mathbf{IR}_\mathcal{A}$. Note that
\[mr((\mathbf{IR}_\mathcal{A})^\top)=\mathbf{MR}_{(\mathfrak{S}_i)_{i\in I}}.\]

Given  $R\in\mathbf{IR}_\mathcal{A}$ 
an interaction relation for $\mathcal{A}$, we  say that
\begin{itemize}
\item $R$ is \emph{normal} if there exists a \emph{normal} interaction request $Q\in \mathbf{MBR}_{(\mathcal{Z},\mathcal{L})}$ such that $\widecheck{Q}=R$,
\item $R$ is \emph{efficient} if $\mathrm{Def}_{br(R)}\subsetneqq\Pi_I\mathcal{Z}$,
\item $R$ is  \emph{functional} (resp. \emph{strongly functional}) if it is so as a request.
\end{itemize}

\begin{exm} For the inclusion order, $\Omega_\mathcal{A}$ is the greatest interaction relations for $\mathcal{A}$. It is normal but not efficient. Indeed, we have $\Omega_\mathcal{A}=\widecheck{Q_M}$, where $Q_M$ designates the greatest interaction request in the context given by $\mathcal{A}$, that is such that $\vert Q_M\vert=\Pi_I\mathcal{E}$, and   $Q_M$ is a normal interaction request since $\mathrm{Def}_{br(Q)}=\Pi_I\mathcal{Z}$. $\Omega_\mathcal{A}$ is not efficient since $\mathrm{Def}_{br(\Omega_\mathcal{A})}=\Pi_I\mathcal{Z}$ (because for all $\sigma_i\in Z_i$, there exists $\lambda_i\in L_i$ such that $\sigma_i\in Z_{i,\lambda_i}$). Note also that, in general, $\Omega_\mathcal{A}$ is not functional. We could say that the graph $\vert \Omega_\mathcal{A}\vert$ is too large to define an efficient interaction relation: interacting is restricting possibilities so, roughly speaking, the smaller is the graph of an interaction relation, the stronger is this interaction.
\end{exm}

\begin{exm}\label{exm interaction relation paranormale entre mono dynamiques}
Let $I=\{1,2\}$ and $A_1=A_2=A$ with $A$ the open functorial non-deterministic mono-dynamic defined by $A=\left((\alpha:(\mathbf{N},+)\rightarrow \mathbf{Tran}) \stackrel{\rho}{\looparrowright}  \mathbf{h}\right)$ with $st(\alpha)=\bullet^\alpha:=\mathbf{N}\times\mathbf{R}$, $st(\mathbf{h})=\mathbf{N}$ 
and for all $(n,r)\in st(\alpha)$,
\begin{itemize}
\item $\rho(n,r)=n$,
\item $\forall d\in\mathbf{N}^*, d^\alpha(n,r)=\{n+d\}\times \mathbf{R}$.
\end{itemize}
The set $Z_A$ of (outgoing) realizations of $A$ can be seen as the set of finite or infinite sequences $\sigma=(s_n)_{n\in \mathbf{N}_\sigma}$ of reals, with $\mathbf{N}_\sigma$ an initial segment of $\mathbf{N}$, and we have $Z_1=Z_2=Z_A$. The set of parameter values of the mono-dynamic $A$ is a singleton, thus we can write $L_1=L_2=\{*\}$. Let's now  consider the interaction relation $R$ given by the graph
\[ \vert R\vert
=
\left\lbrace 
\left(
\begin{array}{cc}
 * & * \\ 
 \sigma & \sigma
\end{array}  
\right),
\sigma\in Z
\right\rbrace.
\]
Then $R$ is obviously a non-normal, functional efficient interaction relation. The lack of normality means that relations between outgoing realizations of the two dynamics $Z_1$ and $Z_2$ are not founded on parameter values, but are directly established. Seeing parameter values as data that dynamics can receive from others, we could say that such a non-normal interaction relation is a ``paranormal" relation.
\end{exm}

\subsection{Synchronizations}\label{subs synchronization}

Recall that $I$ denotes a non-empty set and  $\mathcal{A}=(A_i)_{i\in I}$ an $I$-family 
of 
open dynamics  $A_i=\left((\alpha_i:\mathbf{D}_i\rightharpoondown \mathbf{Tran}^{\underrightarrow{\scriptstyle{L_i}}}) \stackrel{\rho_i}{\looparrowright}  (\mathbf{h}_i:\mathbf{D}_i\rightarrow \mathbf{Sets})\right)$.

Let's begin with the notion of a synchronization of an open dynamic by another, denoting $1$ and $0$ their index in the family $\mathcal{A}$.

\begin{df}\label{df synchronization}
A \emph{synchronization of $A_1$ by $A_0$} is the data $(\Delta,\delta)$ of
\begin{itemize}
\item a map $\Delta:\dot{\mathbf{D}}_0\rightarrow \dot{\mathbf{D}}_1$ defined on the objects of  $\mathbf{D}_0$,
\item a map $\delta:st(\mathbf{h}_0)\rightarrow st(\mathbf{h}_1)$  \emph{compatible} with $\Delta$ in the meaning that
\[\forall S\in\dot{\mathbf{D}}_0, \forall s\in S^{\mathbf{h}_0}, \delta(s)\in (\Delta S)^{\mathbf{h}_1},\]
\end{itemize}
and such that $\delta$ is monotonic, which means that $\delta$ is  
\begin{itemize}
\item either increasing: $\forall(s_0,t_0)\in st(\mathbf{h}_0)^2, s_0\leq_{\mathbf{h}_0} t_0\Rightarrow \delta(s_0)\leq_{\mathbf{h}_1} \delta(t_0)$, 
\item or decreasing:  $\forall(s_0,t_0)\in st(\mathbf{h}_0)^2, s_0\leq_{\mathbf{h}_0} t_0\Rightarrow \delta(t_0)\leq_{\mathbf{h}_1} \delta(s_0)$,
\end{itemize}
where $\leq_{\mathbf{h}_i}$ denotes the pre-order on $\mathbf{h}_i$-instants\footnote{See section \ref{subsubs clocks}.}.
\end{df}

We   write $(\Delta,\delta):\mathbf{h}_0\Rsh \mathbf{h}_1$ to indicate that $(\Delta,\delta)$ is a synchronization of $\mathbf{h}_1$ by $\mathbf{h}_0$. Such a synchronization is said to be \emph{rigid} if $(\Delta,\delta)$ is a (necessarily deterministic) dynamorphism $\mathbf{h}_0\looparrowright \mathbf{h}_1$. Otherwise, it  is called \emph{flexible}\footnote{ In \cite{Dugowson:20150807} and \cite{Dugowson:20150809}, only rigid synchronizations had been considered, while the much more general idea of flexible synchronizations appeared in \cite{Dugowson:20160831}.}.

\begin{df} A \emph{synchronization of the family $\mathcal{A}$ with conductor $i_0\in I$} is a family of synchronizations $((\Delta_i,\delta_i) 
: \mathbf{h}_{i_0}
\Rsh
\mathbf{h}_i)_{i\in I}$, with  $(\Delta_{i_0},\delta_{i_0})=Id_{\mathbf{h}_{i_0}}$.
\end{df}

 \begin{rmq}
 More complex synchronization systems could be usefully considered, which we will not do in this paper.
 \end{rmq}

\subsection{Interactive families}\label{subs familles interactives}

We can now define an \emph{interactive family}
\footnote{
\label{fnote terminology interactive family}
In \cite{Dugowson:20150809}, we used the expression "dynamical families", but this one presents a risk of confusion with the notion of "families of dynamics", and we finally prefer to use the expression ``interactive families".}  
as a family of open dynamics endowed with an interaction request (for example an interaction relation), a family of synchronizations between some of these dynamics and a third element, called \emph{privacy} or social mode, which is an  equivalence relation on the families of parametric values. More precisely: 

\begin{df}\label{df interactive family} 
We call \emph{interactive family} the data\\
\noindent $(I,\mathcal{A}, R,i_0,  (\Delta_i,\delta_i)_{i\in I}, \sim)$ of
\begin{itemize}
\item a non-empty set $I$,
\item an $I$-family $\mathcal{A}=(A_i)_{i\in I}$ of open dynamics, say 
\[A_i=({\rho_i:(\alpha_i:\mathbf{D}_i\rightharpoondown \mathbf{Tran}^{\underrightarrow{\scriptstyle{L_i}}}})\looparrowright \mathbf{h}_i),\]
\item an \emph{interaction $(R,i_0,  (\Delta_i,\delta_i)_{i\in I})$ for $\mathcal{A}$}, that is
\subitem an admissible interaction request $R\in \mathbf{MBR}^\#_{(\mathcal{Z}, \mathcal{L})}$ for $\mathcal{A}$,
\subitem an element $i_0\in I$, 
\subitem a synchronization $((\Delta_i,\delta_i) 
: \mathbf{h}_{i_0}
\Rsh
\mathbf{h}_i)_{i\in I}$ of $\mathcal{A}$ with conductor $i_0$,
\item an equivalence relation $\sim$  on the set $\Pi_I\mathcal{L}=\prod_{i\in I}{L_i}$, called the \emph{intimacy} or the \emph{social mode} of the interactive family.
\end{itemize}
\end{df}

\begin{rmq}
We'll see in section \ref{subs exm of global dynamics} the role of the intimacy of an interactive family.
\end{rmq}

An interactive family with components $\mathcal{A}$ and its interaction  $(R,i_0,  (\Delta_i,\delta_i)_{i\in I})$ are said to be \emph{normal}, \emph{efficient} or \emph{functional} if it is the case for the interaction relation $\widecheck{R}$, respectively.

Let $\mathcal{F}={(I,\mathcal{A}, R,i_0,  (\Delta_i,\delta_i)_{i\in I}, \sim)}$
and $\mathcal{G}=(I,\mathcal{A}, Q,i_0,  (\Delta_i,\delta_i)_{i\in I}, \backsim)$ be two interactive families defined on a same family $\mathcal{A}=(A_i)_{i\in I}$ of open dynamics and sharing a same synchronization $((\Delta_i,\delta_i)_{i\in I}$. If $\widecheck{Q}=\widecheck{R}$, the two interactions $Q$ and $R$ are said to be \emph{strongly equivalent}. If, in addition, the restriction of the equivalence relation $\sim_{\vert M}$ on the set $M=\mathrm{Im}(br(\widecheck{R}))\subset \Pi_I{\mathcal{L}}$ is equal to the restriction $\backsim_{\vert M}$, then the two interactive families $\mathcal{F}$ and $\mathcal{G}$ are said to be \emph{strongly equivalent}.

%\begin{rmq}
%Par \emph{chef d'orchestre} d'une famille interactive, nous entendrons selon les contextes soit l'indice synchronisateur de cette famille, soit la dynamique ouverte dont l'indice dans la famille est l'indice synchronisateur.
%\end{rmq}

\subsection{Connectivity structures of an interactive family}\label{subs connectivity structures of an IF}

Even if we do not discuss in this article the notion of \emph{connective dynamics}\footnote{See \cite{Dugowson:201112}  and  \cite{Dugowson:201203}.}, it should be noted that we have developed the theory of open dynamics and their interactions as an extension of our research on connectivity spaces\footnote{See \cite{Dugowson:201306}.}. This perspective also explains that, regarding topological aspects, we emphasize the connectivity point of view\footnote{For the relation between connectivity and topology, see in particular \cite{Dugowson:20180305}.}. In the present paper we limit ourselves regarding these types of matters to defining a main connectivity structure of an interactive family and three other  connectivity structures, based on its interaction request. At this stage, we do not include in these definitions any considerations about synchronizations.

We begin with some very brief reminders about connectivity spaces and structures and about the connectivity structure of a multiple relation.
 A \emph{connectivity space}\footnote{See \cite{Borger:1983}, \cite{Dugowson:200703} and \cite{Dugowson:201012}.} $X$  is a pair $(\vert X\vert,\kappa(X))$ where $\vert X\vert$ is a set called  the \emph{carrier} of $X$ and $\mathcal{K}=\kappa(X)\subseteq \mathcal{P}(\vert X\vert)$ is called the \emph{connectivity structure} of $X$ and is such that for every $\mathcal{I}\in \mathcal{P}(\mathcal{K})$ we have the implication $\bigcap_{K\in\mathcal{I}}K\ne\emptyset\Rightarrow\bigcup_{K\in\mathcal{I}}K\in\mathcal{K}$. Every element $K\in\mathcal{K}$ is said to be a \emph{connected subset of $\vert X\vert$}, or is simply said to be \emph{connected} (to itself).  When $\vert X\vert$ is non-empty, the empty subset is always connected, because it is the union of the empty family, whose intersection is then non-empty. A connectivity space is said to be \emph{finite} when its carrier is a finite set and it said to be \emph{integral} if every singleton subset is connected. 
  The morphisms between two connectivity spaces are the functions which transform connected subsets into connected subsets.

Given some context $\mathcal{E}=(E_i)_{i\in I}$, \emph{the connectivity space of a multiple relation} $R=(J,\mathcal{E}_{\vert J}, \vert R\vert) \in \mathbf{MR}_{\subseteq\mathcal{E}}$ has been defined in  \cite{Dugowson:20150505} as the space having $J$ as carrier and having as connectivity structure the set $\mathcal{K}_R\subseteq \mathcal{P}(J)$ of subsets $K$ of $J$ that are \emph{non-splittable} for $R$, that is such that there does not exist a partition $K=K_1\sqcup K_2$ with $R_{\vert K}=R_{\vert K_1}\otimes R_{\vert K_2}$.

The notion of connectivity structure of a multiple relation naturally extends to the case of a multiple binary relation: given some context $(\mathcal{W}, \mathcal{M})$ for a given index set $I$,  \emph{the connectivity space of a multiple binary relation} $R=(J,\mathcal{W}_{\vert J}, \mathcal{M}_{\vert J}, \vert R\vert) \in \mathbf{MBR}_{\subseteq(\mathcal{W}, \mathcal{M})}$ is the connectivity space   having $J$ as carrier and having as connected subsets $K\subseteq J$ the ones that are \emph{non-splittable} for $R$, that is such that there does not exist a partition $K=K_1\sqcup K_2$ with $R_{\vert K}=R_{\vert K_1}\otimes R_{\vert K_2}$. In other words, the connectivity structure of a multiple binary relation $R\in \mathbf{MBR}_{\subseteq(\mathcal{W}, \mathcal{M})}$  is the one of the multiple relation $mr(R)\in \mathbf{MR}_{\subseteq\mathcal{E}}$.

For example, given a family $\mathcal{A}=(A_i)_{i\in I}$ of open dynamics, the connectivity structure of the interaction $\Omega=\Omega_\mathcal{A}$ is the discrete integral one\footnote{\label{footnote discrete structure} That is the structure for which the only connected parts of $I$ are the singletons  $\{i\}$ and the empty set, see \cite{Dugowson:201012}.}, because the coherence property is local, that could be written $\Omega=\bigotimes_{i\in I}\Omega_{\vert \{i\}}$.

Given $R\in\mathbf{MBR}_{(\mathcal{Z}, \mathcal{L})}$  an interaction request for a family $\mathcal{A}=(A_i)_{i\in I}$ of open dynamics --- that is\footnote{Using notations of section\,\ref{subs interaction relations in a family}.} :  $\vert R\vert\subseteq \Pi_I \mathcal{E}=\Pi_I(Z_i \times L_i )$ with $Z_i$ the set of outgoing realizations of $A_i$ and $L_i$ the set of its parameter values --- we denote by
$\underline{R}$ the $I$-multiple relation with context $\mathcal{Z}$ obtained by projection (\emph{i.e.} restriction) of $R$ on $\Pi_I \mathcal{Z}$, that is
\[
\vert \underline{R} \vert =
\{
(\sigma_i)_{i\in I}\in \Pi\mathcal{Z}, 
\exists (\lambda_i)_{i\in I}\in\Pi\mathcal{L}, 
\left(
\begin{array}{c}
\lambda_i  \\ 
\sigma_i 
\end{array} 
\right)_{i\in I}\in\vert R\vert
\}.
\]  In the same way, 
 $\underline{\widecheck{R}}$ denotes the projection of $\widecheck{R}$  on $\Pi_I \mathcal{Z}$, where we remind that $\widecheck{R}=R\cap\Omega_\mathcal{A}$ denotes the coherent part of $R$.
Then, we obtain four connectivity structures on $I$ naturally associated with the interaction $R$, that is : 
\begin{itemize}
\item $\mathcal{K}_R$, the connectivity structure of $R\in\mathbf{MBR}_{(\mathcal{Z}, \mathcal{L})}$,
\item  $\mathcal{K}_{\widecheck{R}}$, the connectivity structure of $\widecheck{R}\in\mathbf{MBR}_{(\mathcal{Z}, \mathcal{L})}$, 
\item $\mathcal{K}_{\underline{R}}$, 
the connectivity structure of $\underline{R}\in\mathbf{MR}_{\mathcal{Z}}$,
\item $\mathcal{K}_{\underline{\widecheck{R}}}$ 
the connectivity structure of $\underline{\widecheck{R}}\in\mathbf{MR}_{\mathcal{Z}}$.
\end{itemize}

\begin{prop} For any interaction request 
$R\in\mathbf{MBR}_{(\mathcal{Z}, \mathcal{L})}$ for a given family $\mathcal{A}=(A_i)_{i\in I}$ of open dynamics, we have $\mathcal{K}_{\underline{R}}\subseteq \mathcal{K}_R$ and
$\mathcal{K}_{\underline{\widecheck{R}}}
\subseteq
\mathcal{K}_{\widecheck{R}}
\subseteq
\mathcal{K}_R$.
   Moreover, if $R$ is a normal request, $\mathcal{K}_{\underline{R}}$ is the discrete integral connectivity structure\footnote{See footnote \ref{footnote discrete structure}.}, so in this case, we have
\[
\mathcal{K}_{\underline{R}}
\subseteq
\mathcal{K}_{\underline{\widecheck{R}}}
\subseteq
\mathcal{K}_{\widecheck{R}}
\subseteq
\mathcal{K}_R.
\]
\end{prop}

\begin{proof} 
Let $K\in \mathcal{K}_{\underline{R}}$.  Suppose $K\notin\mathcal{K}_R$: then there is a partition $K=K_1\sqcup K_2$ such that $R_{\vert K}=R_{\vert K_1}\otimes R_{\vert K_2}$. Then
\[
\vert \underline{R}_{\vert K}\vert=
\left\lbrace
(\sigma_k)_{k\in K}\in \Pi_K\mathcal{Z},
\exists  (\lambda_k)_{k\in K}\in \Pi_K\mathcal{L},
\forall n\in\{1,2\},
\left(
\begin{array}{c}
\lambda_k  \\ 
\sigma_k 
\end{array} 
\right)_{k\in K_n}
\in
\vert
R_{\vert K_n}
\vert
\right\rbrace,
\]
so 
$\vert \underline{R}_{\vert K}\vert=
\vert
\underline{R}_{\vert K_1}
\otimes
\underline{R}_{\vert K_2}
\vert$
that is absurd. Thus $K\in\mathcal{K}_R$, so $\mathcal{K}_{\underline{R}}\subseteq\mathcal{K}_R$. The same reasoning applied to $\widecheck{R}$ proves that $ \mathcal{K}_{\underline{\widecheck{R}}}\subseteq\mathcal{K}_{\widecheck{R}}$. 

Now, let's prove that $\mathcal{K}_{\widecheck{R}}
\subseteq\mathcal{K}_R$:
 let $K\in \mathcal{K}_{\widecheck{R}}$, and suppose $K\notin \mathcal{K}_R$. 
 Then, as previously, there is a partition $K=K_1\sqcup K_2$ such that $R_{\vert K}=R_{\vert K_1}\otimes R_{\vert K_2}$. By putting $\Omega=\Omega_\mathcal{A}$, we thus have 
 ${\widecheck{R}_{\vert K}}$ 
 $=R_{\vert K}\cap \Omega_{\vert K}$
 $=(R_{\vert K_1}\otimes R_{\vert K_2})\cap(\Omega_{\vert K_1}\otimes \Omega_{\vert K_2})$. But $\cap$ is nothing but $\otimes$ in the case of a same arity so, by associativity and commutativity, we have 
  ${\widecheck{R}_{\vert K}}$
   $=(R_{\vert K_1}\cap \Omega_{\vert K_1})\otimes( R_{\vert K_2}\cap \Omega_{\vert K_2})$
   $={\widecheck{R}_{\vert K_1}}\otimes{\widecheck{R}_{\vert K_2}}$,  which is absurd, because we assumed that $K\in \mathcal{K}_{\widecheck{R}}$. 
  
Finally, if $R$ is a normal request, then $\vert\underline{R}\vert=\mathrm{Def}_{br(R)}=\Pi_I\mathcal{Z}$, so its connectivity structure is the discrete integral one, that is finer than the others.
\end{proof}

\begin{df}
Given an interactive family
 $\mathcal{F}=(I,\mathcal{A}, R,i_0,  (\Delta_i,\delta_i)_{i\in I}, \sim)$  we call
$\mathcal{K}_\mathcal{F}:=\mathcal{K}_{\underline{\widecheck{R}}}$ the \emph{manifest connectivity structure of $\mathcal{F}$} (or simply the \emph{connectivity structure of $\mathcal{F}$}), and $\mathcal{K}_{\widecheck{R}}$ the \emph{plain connectivity structure of $\mathcal{F}$}.
\end{df}

%%\begin{rmq}
%%Remark that  given an incoming context $\mathcal{W}=(W_i)_{i\in I}$, an outgoing context $\mathcal{M}=(M_i)_{i\in I}$,
%% two subsets $J_1\subseteq I$ and $J_2\subseteq I$  and four multiple binary relations $R_1$, $S_1$, $R_2$ and $S_2$ belonging in $\mathbf{MBR}_{\subseteq(\mathcal{W},\mathcal{M})}$ such that $\mathrm{ar}(R_1)=\mathrm{ar}(S_1)=J_1$ and $\mathrm{ar}(R_2)=\mathrm{ar}(S_2)=J_2$,  we can write
%%\[
%%(R_1\cap S_1)\otimes (R_2\cap S_2)=(R_1 \otimes R_2)\cap(S_1 \otimes S_2)
%%\] because $\cap$ is just a peculiar case of $\otimes$ which is associative and commutative. 
%%\end{rmq} 
 
\subsection{Examples of interactive families}\label{subs examples of interactive families}\label{subs exm of Interactive families}

%\subsubsection{The $\mathbb{WHY}$ = \textcjheb{why} family}
\begin{exm}[The $\mathbb{WHY}$ = \textcjheb{why} family] \label{exm WHY family}
As a first example of an interactive family, let us recall the $\mathbb{WHY}$  family, 
also denoted by \textcjheb{why}, 
that we have introduced in \cite{Dugowson:20160831} and that we have also described in  \cite{DugowsonKlein:20190417}, 
on the occasion of our work with philosopher Pierre Michel Klein concerning his philosophical theory of time, \emph{Metachronology} \cite{KleinPM:2014}. 
As its name suggests, this family involves  the open dynamics $\mathbb{Y}=  \textcjheb{y}$, $\mathbb{H}=\textcjheb{h}$ and $\mathbb{W} =  \textcjheb{w}$ (cf. \emph{supra} examples \ref{exm Iod}, \ref{exm H0}  and \ref{exm Vav}). More precisely, it is defined by 
$\mathbb{WHY}=(I,\mathcal{A}, Q,i_0,  (\Delta_i,\delta_i)_{i\in I}, \sim)$ with
\begin{itemize}
\item $I=\{1,2,3\}$,
\item $\mathcal{A}=(A_i)_{i\in I}$ where $A_1=\mathbb{Y}$, $A_2=\mathbb{H}$ with the origin of times being taken equal to $T_0=0$ for simplicity, and $A_3=\mathbb{W}$, 
\item the graph of the interaction request $Q$ for $\mathcal{A}$ contains all the families
\[
\left(
\begin{array}{ccc}
\lambda_1=* & {\lambda_2} \in L_\mathbb{H}& \lambda_3\in L_\mathbb{W}=\mathcal{C} \\ 
{\sigma_1} \in Z_\mathbb{Y}^* &  \sigma_2\in Z_\mathbb{H}^* & \sigma_3\in  Z_\mathbb{W}^*=\mathcal{C} 
\end{array}	 
\right)
\]
such that $\lambda_3=\sigma_2$ and ${\lambda_2}$ is the restriction of ${\sigma_1}$ to the interior of its domain of definition $\mathrm{Def}_{{\sigma_1}}$, 
\item the conductor is given by $i_0=2$, 
\item $\Delta_1=Id_{\mathbf{R}_+}$ and
$\delta_1:st(\mathbf{h}_\mathbb{H})=]0,+\infty[$ 
$\hookrightarrow$
$[0,+\infty[=st(\mathbf{h}_\mathbb{Y})$, the inclusion map,
\item  $\Delta_3=(\mathbf{R}_+ \stackrel{!}{\rightarrow} \mathbf{1})$ and $\delta_3:st(\mathbf{h}_\mathbb{H})\stackrel{!}{\rightarrow} \{\bullet\}=st(\mathbf{h}_\mathbb{W})$, which is necessarily constant,
\item the social mode $\sim$ (that was not included in our previous definitions of an interactive family) is taken equal to the maximal equivalence relation on $\Pi \mathcal{L}=\{*\}\times L_\mathbb{H}\times L_\mathbb{W}$, \emph{i.e} $\mu \sim \nu$ for all $\mu$ and $\nu$ in $\Pi \mathcal{L}$.
\end{itemize}

Note that the interaction request $Q$ is normal, 
and that the manifest connectivity structure $\mathcal{K}_{\mathbb{WHY}}=\mathcal{K}_{\underline{\widecheck{Q}}}$  is the indiscrete one, that is $\mathcal{K}_{\mathbb{WHY}}=\mathcal{P}(I)$.
\end{exm}

%\subsubsection{The $\mathbb{WH}_1\mathbb{Y}$ family}
%\begin{exm} [The $\mathbb{WH}_1\mathbb{Y}$ family]\label{exm WH1Y family}  We obtain a variant of the interactive family $\mathbb{WHY}$ family taking $\mathbb{H}_1$ instead of $\mathbb{H}$, and keeping everything else the same.
%\end{exm}
%

%\subsubsection{A borromean family}
\begin{exm}[A borromean family]\label{exm a borromean family} Our second example of an interactive family is 
$\mathcal{F}=(I,\mathcal{A}, Q,i_0,  (\Delta_i,\delta_i)_{i\in I},\sim)$ with
\begin{itemize}
\item $I=\{1,2,3\}$,
\item $\mathcal{A}=(A_i)_{i\in I}$ with, for each $i\in I$, $A_i=\Upsilon$, the open dynamic given in the example \ref{exm Upsilon},
\item the graph of the interaction request $Q$ for $\mathcal{A}$ contains all the families
\[
\left(
\begin{array}{ccc}
\lambda_1  & \lambda_2 & \lambda_3 \\ 
\sigma_1   &  \sigma_2 & \sigma_3
\end{array}	 
\right)
\in (Z_\Upsilon \times L_\Upsilon)^3
\]
that satisfy $\{i\in I, \lambda_i(0)=1\}\neq\emptyset$,
\item the conductor is given by $i_0=1$, 
\item  $\delta_i=Id_{\{t_0, t_1\}}$ (and $\Delta_i=Id_{\{T_0, T_1\}}$) for every $i\in I$,
\item the intimacy (social mode) $\sim$ is defined on $L_\Upsilon^3$ by 
\[
(\lambda_1,\lambda_2,\lambda_3)\sim (\mu_1,\mu_2,\mu_3) \Leftrightarrow \lambda_1(0)=\mu_1(0).
\]
\end{itemize}

Note that we can write 
\[
\vert Q\vert
=
\left\lbrace
\left(
\begin{array}{ccc}
\lambda_1  & \lambda_2 & \lambda_3 \\ 
\sigma_1   &  \sigma_2 & \sigma_3
\end{array}	 
\right)
\in (Z_\Upsilon \times L_\Upsilon)^3,
\{\lambda_1, \lambda_2, \lambda_3\}\cap \{\varphi_{10},\varphi_{11}\}\neq\emptyset
\right\rbrace,
\]
where  $\varphi_{kl}$ denotes the map $\{0,1\}\rightarrow\{0,1\}$ such that $\varphi_{kl}(0)=k$ and $\varphi_{kl}(1)=l$, so that we have
$L_\Upsilon=
\left\lbrace 
\varphi_{00},\varphi_{01},\varphi_{10},\varphi_{11}
\right\rbrace$. 

The interaction request $Q$ is obviously normal, and it is easy to see 
that the manifest connectivity structure $\mathcal{K}_{\mathcal{F}}=\mathcal{K}_{\underline{\widecheck{Q}}}$  and the plain connectivity structure $\mathcal{K}_{{\widecheck{Q}}}$ are both the integral borromean one\footnote{Cf. \cite{Dugowson:201012}.}, that is $\mathcal{K}_{\mathcal{F}}=\mathcal{P}(I)\setminus \{\{1,2\},\{2,3\}, \{1,3\}\}$.
\end{exm}

\section{Global dynamics}\label{sec Dyn Produced}

In this section, we associate with any interactive family some global dynamics, \emph{i.e.} some open dynamics produced by the family in question in order to incorporate in a certain way the different dynamics composing the family. 
The difference between these global dynamics lies in the choice of the social mode applied : if it is the social mode belonging to the family itself, we obtain the one we'll call  the global dynamic \emph{demanded} by the considered interactive family. But other choices of a social mode can be made, starting with the trivial equivalence relation (equality) which leads to what we  call the \emph{transparent global dynamic}, on which other global dynamics are modelled.
% and \emph{intrusive interactions}
 Among other possibilities, we  also define the \emph{responsible global dynamic} and the \emph{$J$-global dynamic} --- which results from the choice of the set $J$ of indices for which the parameter values can be determined from the outside --- and we finally introduce the most ``closed" global dynamic (\emph{i.\,e.} a mono-dynamic that cannot be influenced (normally) by some other dynamics), which we  call the \emph{opaque global dynamic} generated by the family.

\subsection{The lax-functorial stability theorem}

%\begin{rmq}
%In the following, we will write some products of sets considering that the order of the factors does not matter, so that we will have, for example  $E_{i_0}\times \prod_{i\neq {i_0}}{E_i}=$
%$(\prod_{i\neq {i_0}}{E_i})\times E_{i_0}$
%$=\prod_{i\in I}{E_i}$.
%\end{rmq}

\begin{thm}[Lax-functorial Stability Theorem]\label{thm stability}  Let $\mathcal{F}=(I,\mathcal{A}, R, i_0, (\Delta_i,\delta_i)_{i\in I})$ be an interactive family, with $\mathcal{A}=(A_i)_{i\in I}$ and, for each $i\in I$
\[A_i=({\rho_i:(\alpha_i:\mathbf{D}_i\rightharpoondown \mathbf{Tran}^{\underrightarrow{\scriptstyle{L_i}}}})\looparrowright \mathbf{h}_i),\] and let 
$\mathbf{E}=\mathbf{D}_{i_0}$ and
$M=\mathrm{Im}(br(\widecheck{R}))$. 
Then we obtain an $M$-dynamic $\beta:$ $\mathbf{E}\rightharpoondown\mathbf{Tran}^{\underrightarrow{\scriptstyle{M}}}$ by putting for every $S\in \dot{\mathbf{E}}$
\[S^\beta
=
\{(a_i)_{i\in I}\in \prod_{i\in I}(\Delta_i S)^{\alpha_i},
\forall i\in I, \rho_{i}(a_i)=\delta_i(\rho_{{i_0}}(a_{i_0}))\},\]  and, for every $(d:S\rightarrow T)\in \overrightarrow{\mathbf{E}}$,  $a=(a_i)_{i\in I}\in S^\beta$ and  $\mu\in M$, by defining 
$d^\beta_\mu(a)$ as the set of the $b=(b_i)_{i\in I}\in T^\beta$ such that  
\begin{equation}\label{eq 2 real}
\exists (\sigma_i)_{i\in I}\in  br(\widecheck{R})^{-1}(\mu),
\forall i\in I, \sigma_i\triangleright a_i, b_i
\end{equation}
and
\begin{equation}\label{eq 1 sync}
\rho_{{i_0}}(b_{i_0})=d^{\mathbf{h}_{i_0}}(\rho_{{i_0}}(a_{i_0})).
\end{equation}
\end{thm}

\begin{proof}
First, we have $M\neq\emptyset$, because $R$ is an admissible request so $P=\widecheck{R}$ and $M=\mathrm{Im}(br(P))$ are not empty. Then, following the section §\,\ref{subsubs LdynD}, we have to check that $\beta$ is a disjunctive lax-functor, \emph{i.e.} that these three conditions are satisfied:
\begin{enumerate}
\item (Disjunctivity) $\forall(S,T)\in\dot{\mathbf{E}}^2$, $S\neq T \Rightarrow S^\beta \cap T^\beta=\emptyset$,
\item (Lax identity) $\forall S\in\dot{\mathbf{E}}, \forall \mu\in M, (Id_S)^\beta_\mu\subseteq Id_{S^\beta}$,
\item (Lax composition) for every $(S{\stackrel{d}{\rightarrow}}T{\stackrel{e}{\rightarrow}}U)$ in $\mathbf{E}$ and every  $\mu\in M$, 
\[(e\circ d)^\beta_\mu\subseteq e^\beta_\mu \odot d^\beta_\mu.\]
\end{enumerate}

\subparagraph{1. Disjunctivity.} Suppose $S\neq T$  but $S^\beta\cap T^\beta\neq\emptyset$, then we would have an element
$(a_i)_{i\in I}\in S^\beta\cap T^\beta$ and  --- as $\Delta_{i_0}=Id_{\mathbf{D}_{i_0}}$ and $\alpha_{i_0}$ is disjunctive --- we would have $a_{i_0}\in S^{\alpha_{i_0}}\cap T^{\alpha_{i_0}}=\emptyset$ , that is absurd.

\subparagraph{2. Lax identity.} Let $S\in\dot{\mathbf{E}}$ and $\mu\in M$.
We want to check that $(Id_S)^\beta_\mu\subseteq Id_{S^\beta}$. In other words, we want to check that if $S^\beta\neq\emptyset$ and $a={(a_i)_{i\in I}}\in S^\beta$ then $(Id_S)^\beta_\mu(a)\subseteq \{a\}$, that is $(Id_S)^\beta_\mu(a)=\emptyset$ or $(Id_S)^\beta_\mu(a)=\{a\}$. But if $(Id_S)^\beta_\mu(a)$ is not empty and $a'={(a'_i)_{i\in I}}$ is an element of it, then 
for every $i\in I$, there is an outgoing realization $\sigma_i\in Z_{i}$ such that $\sigma_i\rhd a_i,a'_i$ and then --- using  $\sigma_i\rhd a'_i$, the definition of $S^\beta$, the  condition (\ref{eq 1 sync}) and $\sigma_i\rhd a_i$ ---
we have
\[
a'_i
= \sigma_i(\rho_i(a'_i))
=\sigma_i(\delta_i(\rho_{i_0}(a'_{i_0})))
=\sigma_i(\delta_i(\rho_{i_0}(a_{i_0})))
= \sigma_i(\rho_i(a_i))
=a_i.
\]
Thus $a'=a$, and we have proved that
$(Id_S)^\beta_\mu\subseteq Id_{S^\beta}$.

\subparagraph{3. Lax composition.} We have to check that given any $(S{\stackrel{d}{\rightarrow}}T{\stackrel{e}{\rightarrow}}U)$ in $\mathbf{E}$, any
$\mu\in M$ and any state $a=(a_i)_{i\in I}\in S^\beta$, we have
$(e\circ d)^\beta_\mu(a)\subseteq (e^\beta_\mu \odot d^\beta_\mu) (a)
$. In other words, supposing $(e\circ d)^\beta_\mu(a)$ not empty and taking any $c=(c_i)_{i\in I}\in(e\circ d)^\beta_\mu(a)\subseteq U^\beta$, we have to prove the existence of a state $b\in d^\beta_\mu(a)\subseteq T^\beta$ such that $c\in e^\beta_\mu(b)$.

To express such a state $b=(b_i)_{i\in I}$, let us set $t_0:=\rho_{i_0}(a_{i_0})\in S^{\mathbf{h}_{i_0}}$, 
 $t_1:=d^{\mathbf{h}_{i_0}}(t_0)\in T^{\mathbf{h}_{i_0}}$,
and $t_2:=e^{\mathbf{h}_{i_0}}(t_1)\in U^{\mathbf{h}_{i_0}}$.
Note that by the definition of $(e\circ d)^\beta_\mu(a)$ we also have $t_2=\rho_{i_0}(c_{i_0})$ and that there exists a family $(\sigma_i)_{i\in I}\in br(\widecheck{R})^{-1}(\mu)$ of outgoing realizations such that $\sigma_i \rhd a_i,c_i$ for every $i\in I$. Given $(\sigma_i)_{i\in I}$ such a family, it suffices now to prove that each $\sigma_i$ is defined for the instant $\delta_i(t_1)$, and that $b_i:= \sigma_i(\delta_i(t_1))$ is a suitable choice.
Note first that since ${\sigma_{i_0} \rhd c_{i_0}}$, we have $\rho_{i_0}(c_{i_0})\in \mathrm{Def}_{\sigma_{i_0}}$, that is $t_2\in\mathrm{Def}_{\sigma_{i_0}}$. But $t_1\leq_{\mathbf{h}_{i_0}}t_2$, 
so, according to the properties of realizations (see section §\,\ref{subsubs realizations of an open dynamic}),
 $\delta_{i_0}(t_1)=t_1\in \mathrm{Def}_{\sigma_{i_0}}$. 
 
Let us now consider the case of an $i\neq {i_0}$. By definition of a synchronization, the map  $\delta_i$ is either increasing or decreasing. If it is increasing, then $\delta_i(t_1)\leq_{\mathbf{h}_i} \delta_i(t_2)$, but $c_i=\sigma_i(\rho_i(c_i))=\sigma_i(\delta_i(t_2))$, so
 $\delta_i(t_1)\in \mathrm{Def}_{\sigma_{i}}$. If $\delta_i$ is decreasing, then  $\delta_i(t_1)\leq_{\mathbf{h}_i} \delta_i(t_0)$, but $a_i=\sigma_i(\rho_i(a_i))=\sigma_i(\delta_i(t_0))$, so $\delta_i(t_0)\in \mathrm{Def}_{\sigma_{i}}$  and thus we have again $\delta_i(t_1)\in \mathrm{Def}_{\sigma_{i}}$. Now, let's put $b_i=\sigma_i(\delta_i(t_1))$ for every $i\in I$.
Then $b\in d^\beta_\mu(a)$, since
\begin{itemize}
\item for every $i\in I$, $t_1\in T^{\mathbf{h}_{i_0}}\Rightarrow \delta_i(t_1)\in (\Delta_i T)^{\mathbf{h}_i}$, and then $b_i=\sigma_i(\delta_i(t_1))\in  (\Delta_i T)^{\alpha_i}$,
\item by definition of a realization $\rho_{i_0}(b_{i_0})=\rho_{i_0}(\sigma_{i_0}(t_1))=t_1$ and, for every  $i\in I$,
$
\rho_i(b_i)
=\rho_i(\sigma_i(\delta_i(t_1)))
=\delta_i(t_1)
=\delta_i(\rho_{i_0}(b_{i_0}))$,
\item by construction, we have $\sigma_i\rhd a_i, b_i$ for every $i\in I$.
\end{itemize}

But we also have $c\in e^\beta_\mu(b)$, since
\begin{itemize}
\item $c\in U^\beta$,
\item $\rho_{i_0}(c_{i_0})=t_2=(e\circ d)^{\mathbf{h}_{i_0}}(t_0)=e^{\mathbf{h}_{i_0}}(t_1)=e^{\mathbf{h}_{i_0}}(\rho_{i_0}(b_{i_0}))$,
\item and for all $i\in I$, $\sigma_i\rhd b_i,c_i$,
\end{itemize}
and this concludes the proof.
\end{proof} 

\subsection{The transparent global dynamic}

Thanks to the theorem \ref{thm stability}, it is immediate to check that the definition below is consistent.  

%Because for every  $S\in \dot{\mathbf{D}}_{i_0}$ and every $a\in S^\beta$, one has $\rho(a)\in S^\mathbf{h}$,
%and for every
%$(S{\stackrel{d}{\rightarrow}}T)\in\overrightarrow{\mathbf{D}_{i_0}}$, every $\mu\in M$, every $a=(a_i)_{i\in I}\in S^\beta$ and every  $b\in d^\beta_\mu(a)$,  one has
%\[\rho(d^\beta_\mu(a))
%=\tau_{i_0}(b_{i_0})
%=d^{\mathbf{h}_{i_0}}(\tau_{{i_0}}(a_{i_0}))
%=d^{\mathbf{h}}(\rho(a)). 
% \]

\begin{df} Using the same notations than above, the \emph{transparent global dynamic associated with an interactive family $\mathcal{F}$} is the open dynamics denoted $[\mathcal{F}]_1$ defined by
\[
[\mathcal{F}]_1= 
\left(
(
\beta:\mathbf{E}\rightharpoondown\mathbf{Tran}^{\underrightarrow{\scriptstyle{M}}}
)
\stackrel{\tau}{\looparrowright}  
(\mathbf{k}:\mathbf{E}\rightarrow \mathbf{Sets})
\right)
\]
where $\beta$ and thus, in particular, $M$ and $\mathbf{E}$, are the one associated with $\mathcal{F}$ by the theorem \ref{thm stability}, the clock $\mathbf{k}$ is given by $\mathbf{k}=\mathbf{h}_{i_0}$ and the datation 
 $\tau:st(\beta)\rightarrow st(\mathbf{k})$ is defined by
\[\forall S\in\dot{\mathbf{E}}, \forall a=(a_i)_{i\in I}\in S^\beta, \tau(a)=\rho_{i_0}(a_{i_0}).\] 
\end{df}

%Afin de \emph{réduire} l'ensemble paramétrique, nous pouvons faire appel  à une relation d'équivalence , choisie aussi judicieusement que possible,  pour former, conformément à la définition \ref{df reduc param} (page \pageref{df reduc param}), la dynamique sous-fonctorielle ouverte quotient $[\mathcal{F}]/{\sim}$. En prenant pour relation $\sim$ la relation totale sur $M$, nous obtiendrons pour $[\mathcal{F}]/{\sim}$ une mono-dynamique `` ouverte'', que nous noterons ${[\mathcal{F}]_\mathrm{m}}$, dont les transitions ne dépendent donc d'aucun paramètre. 
%Dès lors il s'agit de
%\begin{quote}
%[...] trouver le juste équilibre entre l'ouverture excessive de ${[\mathcal{F}]}$, offrant des paramètres en réalité souvent inutilisables, et la fermeture complète sur elle-même de ${[\mathcal{F}]_\mathrm{m}}$, qui n'offre plus aucune prise à l'interaction%
%\footnote{Du moins si on se limites aux interactions normales, voir plus la haut la section \textbf{§\,\ref{subsubs interactions normales}}.
%}  %
% avec d'autres dynamiques.\footnote{Extrait de la remarque 12 dans \cite{Dugowson:20150809}.}
%\end{quote}

\subsection{The demanded global dynamic}\label{subs demanded global dynamics}

The parametric set $M$ of the transparent global dynamic $[\mathcal{F}]_1$ associated with an interactive family $\mathcal{F}$ is generally ``too big"  in the sense that very often some parts of the parametric values are not intended to be externally controlled and should instead be determined by the realizations of the dynamics that compose the interactive family itself. The social mode (or intimacy) $\sim$ of $\mathcal{F}$ --- which does not play a role in the definition of the transparent global dynamic --- is precisely used to ``reduce" the parametric set, thanks to the notion of the parametric quotient of an open dynamic by an equivalence relation on the set of parametric values (see definition \ref{df reduc param}). The response of the new global dynamic thus obtained is the same for two distinct parametric values, as long as they are equivalent: it is up to it to take into account, or not, the requests made to it from outside, by constructing its response on all the possibilities given to it by the different equivalent parametric values of a same equivalence class. 

\begin{df} 
Using the same notations than above, and denoting again $\sim$ the restriction of the intimacy $\sim$ of $\mathcal{F}$ to the subset $M=\mathrm{Im}(br(\widecheck{R}))\subseteq \Pi_I\mathcal{L}$, the  global dynamic \emph{demanded} by $\mathcal{F}$ --- also called the \emph{demanded global dynamic of $\mathcal{F}$} --- is defined as the open dynamic denoted $[\mathcal{F}]_{\sim}$ given by
$
[\mathcal{F}]_{\sim}= [\mathcal{F}]_{1}/{\sim}.
$
\end{df}

\subsection{The responsible global dynamic}\label{subs functional global dynamic}

In this section, we associate with any interactive request an intimacy called ``responsible intimacy" that intuitively allows the interactive family to choose the parametric values of each dynamic at stake when these values are susceptible to be determined by the realizations of the other dynamics of the family. More precisely, using the same notations as previously, if $Q$ designates a (not necessarily coherent) admissible interaction request 
for a family $\mathcal{A}=(A_i)_{i\in I}$ of open dynamics, we define the responsible intimacy  $\asymp_Q$  for $Q$ as the equivalence relation on $\Pi_I\mathcal{L}$ setting, for any  $((\mu_i)_{i\in I},(\lambda_i)_{i\in I})\in (\Pi_I\mathcal{L})^2$,
${(\mu_i)_{i\in I}\asymp_Q(\lambda_i)_{i\in I}}$ iff we have, for all $i\in I$: ($\mu_i=\lambda_i$ or $\mu_i\in N_i\ni\lambda_i$), where $N_i\subseteq L_i$ is defined as the set 
\[
N_i:=
\left\lbrace
l\in L_i, 
\forall (\mathfrak{p},\mathfrak{q})\in\vert Q\vert^2,
\left(
\mathfrak{p}^i=l
\,\mathrm{and}\,
\forall k\neq i, \mathfrak{p}_k=\mathfrak{q}_k
\right)
\Rightarrow 
\mathfrak{q}^i=l
\right\rbrace.
\]

The \emph{responsible global dynamic}  $[\mathcal{F}]_{\asymp_Q}$   generated by an interactive family $\mathcal{F}=(I,\mathcal{A}, Q,$ $i_0,  (\Delta_i,\delta_i)_{i\in I}, \sim)$ is then defined as the demanded global dynamic of the interactive family $(I,\mathcal{A}, Q,i_0,  (\Delta_i,\delta_i)_{i\in I}, \asymp_Q)$. Note that  $[\mathcal{F}]_{\asymp_Q}$ does not depend on the social  mode $\sim$ demanded by $\mathcal{F}$ itself, and that two different interaction requests $Q$ and $R$ can result in two different social modes ${\asymp_Q}$ and ${\asymp_R}$ even if $\widecheck{Q}=\widecheck{R}$.

\subsection{The $J$-global dynamic}\label{subs J-global dynamics} Let $\mathcal{F}=(I,\mathcal{A}, Q,$ $i_0,  (\Delta_i,\delta_i)_{i\in I}, \sim)$  be an interactive family as previously, and let $J\subseteq I$ a subset whose elements $j\in J$ intuitively represent the indices such that the corresponding dynamics $A_j$  could be influenced from outside the global dynamic we want to define. To achieve this, we can use a similar construction to the one we described for the responsible global dynamic, but taking for each $i\in I$ the set $N_i$ given by: $N_i=\emptyset$ if $i\in J$ and $N_i=L_i$ if $i\notin J$. In this way, we get a global dynamics $[\mathcal{F}]_{\sim_J}$ that we'll call the $J$-global dynamic associated with $\mathcal{F}$. When $J=I$, we obtain the transparent global dynamic  $[\mathcal{F}]_{1}=[\mathcal{F}]_{\sim_I}$   associated with $\mathcal{F}$.

\subsection{The opaque global dynamic}\label{subs required global dynamics}

The transparent global dynamic $[\mathcal{F}]_1$ is the ``most open" of the global dynamics associated with an interactive family (so much so that it is generally ``too open"). At the other end, the \emph{opaque global dynamic} is the most ``closed" of them, since its parametric set is reduced to a singleton. It is obtained by making the quotient of $[\mathcal{F}]_1$ by the maximum equivalence relationship on $M$, for which $M$ is the only equivalence class. Denoting again $M$ this equivalence relation, we thus have:

\begin{df} 
The \emph{opaque global dynamic associated with $\mathcal{F}$} is the open mono-dynamic denoted $[\mathcal{F}]_{0}$ defined by
$
[\mathcal{F}]_{0}= [\mathcal{F}]_{1}/{M}.
$
\end{df}

In other words, the opaque global dynamic generated by $\mathcal{F}$ is its $\emptyset$-global dynamic: $[\mathcal{F}]_{0}=[\mathcal{F}]_{\sim_\emptyset}$. 
The following proposition immediately follows  from the definitions:
\begin{prop}
If the interaction request of an interactive family is strongly functional, then its responsible global dynamic is the opaque one.
\end{prop}

\begin{rmq}
With a strongly functional interaction request, the responsible global dynamic (which in this case is the opaque one) generated by a family can be non-deterministic even if all the dynamics at stake are deterministic  (for example if the request is that each dynamic has a similar behavior than another, then each one can be entirely determined by the others without the responsible global dynamic being deterministic). This shows a certain instability of determinism.
\end{rmq}

\subsection{Examples of global dynamics}\label{subs exm of global dynamics}

The examples in this section are the global dynamics associated with the interactive families proposed as examples in the section §\,\ref{subs examples of interactive families}.

\begin{example}[The global dynamic of the $\mathbb{WHY}$ family]
 Like in the example \ref{exm Vav}, given two (partial) functions $a$ and $b$ on $\mathbf{R}$, we use the notation $a\diamondsuit b$ to say that $a_{\vert \mathrm{Def}_a\cap \mathrm{Def}_b}=b_{\vert \mathrm{Def}_a\cap \mathrm{Def}_b}$. Then one checks that 
 the interactive family $\mathbb{WHY}$  of the example \ref{exm WHY family} generates as global dynamic the opaque global dynamic $\mathbb{S}=[\mathbb{WHY}]_\sim$ given by
\[
\mathbb{S}=
\left(
(
\alpha_\mathbb{S}:\mathbf{D}_\mathbb{S}\rightharpoondown\mathbf{Tran}
)
\stackrel{\tau}{\looparrowright}  
\mathbf{h}_\mathbb{S}
\right)
\]
where
\begin{itemize}
\item $\mathbf{D}_\mathbb{S}=\mathbf{D}_\mathbb{H}=(\mathbf{R}_+,+)$,
\item $\mathbf{h}_\mathbb{S}=\mathbf{h}_\mathbb{H}$, so  $st(\mathbf{h}_\mathbb{S})=]0,+\infty[$ and for all $t\in ]0,+\infty[$  and all $d\in\mathbf{R}_+$, we have $d^{\mathbf{h}_\mathbb{S}}(t)=t+d$,
\item $st(\mathbb{S})=
\{(t,r,f,w)
\in 
\mathbf{R}_+^*
\times 
\mathbf{R}
\times
\mathcal{C}^1
\times
\mathcal{C}, 
\mathrm{Def}_f=]-\infty, t[
\}$,
\item $\forall (t,r,f,w)\in st(\mathbb{S}), \tau(t,r,f,w)=t$,
\end{itemize}
and such that a state $(t,r,f,w)\in st(\mathbb{S})$ is onside iff $r=f(t^-):=\lim_{s\rightarrow t^-}f(s)$, $f_{\vert [0,t[}\in Lip_1([0,t[)$ and $f\diamondsuit w$, and that in this case the set $d^{\alpha_\mathbb{S}}(t,r,f,w)$ is given for any duration $d\in\mathbf{R}_+$ as the set of all states $(t+d,q,g,w)$ such that 
\begin{itemize}
\item  $g_{\vert ]-\infty,t[}=f$,
\item $q=g((t+d)^-)$,
\item $g_{\vert [0,t+d[}\in Lip_1([0,t+d[)$,
\item $g \diamondsuit w$.
\end{itemize}

\end{example} 

%
%\paragraph{Continued from example \ref{exm WH1Y family}.} The global dynamic  $\mathbb{S}_1=[\mathbb{WH}_1\mathbb{Y}]_\sim$  generated by the interactive family of the example \ref{exm WH1Y family}  is the same as the dynamic $[\mathbb{WHY}]_\sim$ described above, except that it states are the $(t, r, f, w)\in \mathbf{R}_+^*\times \mathbf{R}\times\mathcal{C}\times\mathcal{C}$ such that $f\in\mathcal{C}(]-\infty, t[)$, and then the set  $d^{\alpha_{\mathbb{S}_1}}(t,r,f,w)$ is the one of all states $(t+d,q,g,w)$ such that  
%$f(t^-)=r$,
%$w \diamondsuit g \diamondsuit f $ 
%(that is $w \diamondsuit g$ and $f= g_{\vert ]-\infty,t[}$),
%$g_{\vert ]0,t+d[}\in Lip_1(]0,t+d[)$,
%$\lim_{s\rightarrow (t+d)^-}g(s)=q$ (and thus $\vert r-q\vert\leq d$),
%
%
%if such a state is onside, that is if $f(t^-)=r$, $f\diamondsuit w$ and $f_{\vert [0,t[}\in Lip_1([0,t[)$, 
%

\begin{example}[The global dynamic of a borromean family]
One checks that the borromean interactive family considered in the example \ref{exm a borromean family} results in a functorial global dynamics isomorphic (in $\mathbf{ODyn}$) to the one given by

\[
\mathbb{U}=
\left(
(\mathbf{u}:\mathbf{D}_\Upsilon\rightarrow \mathbf{Tran}^{\underrightarrow{\scriptstyle{M_\mathbb{U}}}}) 
\stackrel{!}{\looparrowright}  
(\zeta_{\mathbf{D}_\Upsilon}:{\mathbf{D}_\Upsilon}\rightarrow \mathbf{Sets})
\right)
\]
where we recall that  $\mathbf{D}_\Upsilon$ is one-step category $(T_0 \stackrel{d}{\rightarrow} T_1)$ and $\zeta_{\mathbf{D}_\Upsilon}$ is its \emph{essential clock} with instants ${T_k}^{\zeta_{\mathbf{D}_\Upsilon}}=\{t_k\}$ (cf. example \ref{exm Upsilon}), and with
\begin{itemize}
\item $\forall k\in\{0,1\}$, ${(T_k)}^\mathbf{u}=\{t_k\}\times \{0,1\}^3$, 
\item $M_\mathbb{U}=\{0,1\}$,
\item $\forall \mu\in M_\mathbb{U}$,
 $\forall k\in\{0,1\}$, $({Id_{(T_k)}})^\mathbf{u}_\mu=Id_{({(T_k)}^\mathbf{u})}$ (since ${\mathbf{u}}$ is functorial), 
\item for $\mu\in M_\mathbb{U}$ and $(a,b,c)\in\{0,1\}^3$, $d^\mathbf{u}_\mu(t_0;a,b,c)$ is the set of all states of the form $(t_1;a',b',c')\in (T_1)^\mathbf{u}$ such that
\subitem if ($\mu=0$ and $a=0$ and $(b,c)\neq(0,0)$) then $a'=0$,
\subitem if ($\mu=0$ and $(a,b,c)=(0,0,0)$) then ($a'=0$ and ($b'=1$ or $c'=1$)),
\subitem if ($\mu=0$ and $(a,b,c)=(1,0,0)$) then ($b'=1$ or $c'=1$),
%\subitem if ($\mu=0$ and $a=1$ and $(bc)\neq(00)$) then $(a'b'c')\in \{0,1\}^3$,
\subitem if ($\mu=1$ and $a=0$) then $a'=1$,
%\subitem if ($\mu=1$ and $a=1$) then $(a'b'c')\in \{0,1\}^3$,
\item  $\mathbf{u} \stackrel{!}{\looparrowright} \zeta_{\mathbf{D}_\Upsilon}$ is the  unique possible deterministic dynamorphism, for which the date of $(t_k;a,b,c)$ is $t_k$.
\end{itemize}
\end{example}

\section*{Conclusion}\label{sec conclusion}\addcontentsline{toc}{section}{Conclusion}

This article presents the basics of our theory of interacting open dynamics, but various important questions are not addressed at all. In particular, while we have presented the connectivity structures of interactions, a subsequent article will have to take up the theme of the connectivity structures of the dynamics themselves (theme that we had addressed in \cite{Dugowson:201112} and \cite{Dugowson:201203} in the case of mono-dynamics), in order to clarify the relationships between the connectivity structures of the dynamics of an interactive family, the connectivity structure of the interaction and the connectivity structure of the global dynamic generated by such a family.

In addition, the study of the relationships between our theory and Andrée Ehresmann's suggests addressing some ideas that are currently absent from our theory, in particular the question of how an interactive family can continue to produce a global dynamic when certain dynamics of this family cease to function, or when new dynamics enter the dance. More generally, we hope to address the question of self-organization, which is currently largely absent from our theory. On the other hand, as we have seen, the study of the relationships between the two theories suggests a non-deterministic extension of Andrée Ehresmann's guidable systems, which should be studied. 
Furthermore, it would be interesting to clarify the relation with other ``compositional theories" such as David Spivak's dynamical theory.

\paragraph{Acknowledgement.} Special thanks to Mathieu Anel, Spencer Breiner, Olivia Caramello, Jean-Yves Choley, Andrée Ehresmann, René Guitart, Anatole Khélif, Pierre Michel Klein, Marc Lachièze-Rey, René Lew, Marie-Odile Monchicourt, François Nicolas, Alain Rivière, David Spivak, Jean-Jacques Szczeciniarz, Susan Tro\-meur, Marina Ville, Noson S. Yanofsky, Patrick-Iglesias Zemmour. There are many other colleagues, friends and relatives who, in one way or another, have helped me directly or indirectly in this work: may they also be thanked.

% List of references
%\bibliographystyle{plain}
%\bibliography{allreferences-url-var}

%\thispagestyle{empty}

% Cahiers wants the author's address at the end of the paper:

\vspace{5mm}
\noindent
Stéphane Dugowson \\
Institut Supérieur de Mécanique de Paris (Supméca) \\
3, rue Fernand Hainaut \\
93407 St Ouen (France) \\
s.dugowson@gmail.com

% Done!
\newpage
\section*{Contents}
\noindent{Introduction}\\
\mbox{}\quad  {Notations and $2$-categories at stake}\\
  \mbox{}\quad\quad{Functions}\\
   \mbox{}\quad\quad{Categories}\\
   \mbox{}\quad\quad{Transitions}\\
   \mbox{}\quad\quad{The $2$-categories $\mathbf{Tran}$ and $\mathbf{ParF}$}\\
   \mbox{}\quad\quad{Some $2$-categories of sets with $L$-families of transitions as arrows}\\
  
 \noindent{1. Open dynamics}\\
   \mbox{}\quad{1.1 Multi-dynamics}\\
   \mbox{}\quad\quad{1.1.1. $L$-dynamics on a category $\mathbf{D}$}\\
   \mbox{}\quad\quad{1.1.2 The category $\mathbf{MonoDyn}_\mathbf{D}$ of mono-dynamics on $\mathbf{D}$}\\
   \mbox{}\quad\quad{1.1.3. Clocks on $\mathbf{D}$}\\
   \mbox{}\quad\quad{1.1.4. The category $L-\mathbf{Dyn}_\mathbf{D}$ of $L$-dynamics on $\mathbf{D}$}\\
   \mbox{}\quad\quad{1.1.5. The category $\mathbf{MultiDyn}_\mathbf{D}$ of multi-dynamics on $\mathbf{D}$}\\
   \mbox{}\quad\quad{1.1.6. The category $\mathbf{MultiDyn}$ of multi-dynamics}\\
  \mbox{}\quad {1.2. Open dynamics: definition, realizations, quotients}\\
   \mbox{}\quad\quad{1.2.1. Definition of open dynamics}\\
   \mbox{}\quad\quad{1.2.2 Realizations of an open dynamic}\\
   \mbox{}\quad\quad{1.2.3. Parametric quotients}\\
   \mbox{}\quad{1.3. Examples of open dynamics}\\
   \mbox{}\quad{1.4. Some relations with Bastiani (-Ehresmann)'s control systems}\\

 \noindent{2.Interactive families}\\
   \mbox{}\quad{2.1. Binary, multiple and multiple binary relations}\\
   \mbox{}\quad\quad{2.1.1. Binary relations}\\
   \mbox{}\quad\quad{2.1.2. Multiple relations}\\
   \mbox{}\quad\quad{2.1.3. Multiple binary relations}\\
  { \mbox{}\quad\quad 2.1.4. Type conversions between $\mathbf{MBR}_I$, $\mathbf{MR}_{2I}$, $\mathbf{MR}_I$ and $\mathbf{BR}$}\\
 \mbox{}\quad  {2.2. Interaction relations in a family of open dynamics}\\
 \mbox{}\quad  {2.3. Synchronizations}\\
 \mbox{}\quad  {2.4. Interactive families}\\
 \mbox{}\quad  {2.5. Connectivity structures of an interactive family}\\
 \mbox{}\quad  {2.6. Examples of interactive families}\\

 \noindent{3. Global dynamics}\\
  \mbox{}\quad {3.1. The lax-functorial stability theorem}\\
   \mbox{}\quad  {3.2. The transparent global dynamic}\\
  \mbox{}\quad{3.3. The demanded global dynamic}\\
 \mbox{}\quad  {3.4. The responsible global dynamic}\\
 \mbox{}\quad  {3.5. The $J$-global dynamic}\\
 \mbox{}\quad  {3.6. The opaque global dynamic}\\
 \mbox{}\quad  {3.7. Examples of global dynamics}\\
  
 \noindent{Conclusion}\\
 
 \noindent{References}\\
 
 \noindent{Contents}\\

\end{document}